\documentclass[reqno]{amsart}
\usepackage{graphicx,a4wide}
\usepackage[utf8]{inputenc}
\usepackage{amssymb,epsfig,amsmath,latexsym,enumerate,enumitem,mathtools,bbm,tikz,hyperref,amsfonts,amsthm,url,comment,orcidlink,caption,tikz-cd}
\usetikzlibrary{patterns,arrows,automata,shapes,chains,decorations.pathreplacing,calligraphy}
\usetikzlibrary[calc]

\newcommand{\1}{\mathbbm{1}}

\newcommand{\N}{\mathbb{N}}

\newcommand{\R}{\mathbb{R}}

\newcommand{\BBB}{\mathcal{B}}

\newcommand{\AAA}{\mathcal{A}}

\newcommand{\JJJ}{\mathcal{J}}
\newcommand{\OOO}{\mathcal{O}}
\newcommand{\UUU}{\mathcal{U}}

\newcommand{\ltri}{\vartriangleleft}
\newcommand{\ltrieq}{\trianglelefteq}
\newcommand{\rtri}{\vartriangleright}
\newcommand{\rtrieq}{\trianglerighteq}
\newcommand{\eq}{\text{eq}}
\newcommand{\uns}[1]{\breve{S}_{#1}}
\newcommand{\ac}[1]{\chi_{#1}}
\newcommand{\oic}[1]{\hat{\chi}_{#1}}
\newcommand{\unse}[2]{s_{(#1,#2)}}
\newcommand{\wh}{\widehat}
\newcommand{\pl}{\text{PL}}
\newcommand{\bi}{\text{B}}
\DeclareMathOperator{\Span}{Span}
\DeclareMathOperator{\pro}{pro}
\DeclareMathOperator{\RR}{R}
\newtheorem{theorem}{Theorem}[section]

\newtheorem{lemma}[theorem]{Lemma}
\newtheorem{cor}[theorem]{Corollary}

\newtheorem{conj}[theorem]{Conjecture}
\newtheorem{ques}[theorem]{Question}
\theoremstyle{remark}
\newtheorem*{remark}{Notation}
\theoremstyle{definition}
\newtheorem{ex}[theorem]{Example}

\newcounter{comments}
\tikzset{fill lower half/.style={path picture={\fill[#1] (path picture bounding box.south west)
  rectangle (path picture bounding box.east);}}}
\tikzstyle{fr} = [rectangle, draw, fill=black]
\tikzstyle{hr} = [rectangle, draw, fill lower half=black]
\tikzstyle{er} = [rectangle, draw]
\newcommand{\makerecs}[3]{\begin{tikzpicture}
\node (1) at (0,0) [#1] {};
\node (2) at (0.3,0) [#2] {};
\node (3) at (0.6,0) [#3] {};
\end{tikzpicture}}


\begin{document}

\title{Toggleability Spaces of Fences}
\author{Alec Mertin \orcidlink{0009-0009-3167-7501}}
\email{amertin@clemson.edu}
\address{School of Mathematical and Statistical Sciences\\ Clemson University\\ Clemson\\ SC\\ U.S.A.}

\author{Svetlana Poznanovi\'c \orcidlink{0000-0002-8229-8220}}
\email{spoznan@clemson.edu}
\address{School of Mathematical and Statistical Sciences\\ Clemson University\\ Clemson\\ SC\\ U.S.A.}

\date{\today}

\begin{abstract}
We completely describe the \emph{order ideal (\textup{resp.} antichain) toggleability space} for general fences: the space of statistics which are linear combinations of order ideal (antichain) indicator functions and equal to a constant plus a linear combination of \emph{toggleability statistics}. This allows us to strengthen some homomesies under rowmotion on fences proven by Elizalde et al. and prove some new homomesy results for combinatorial, piecewise-linear, and birational rowmotion.
\end{abstract}

\maketitle


\section{Introduction}
\label{intro section}

The work in this paper is motivated by the goal of understanding the occurrence of the homomesy phenomenon under the action of rowmotion on order ideals of fences within two natural subspaces of statistics. \emph{Homomesy} occurs when a statistic on a set of objects has the same average along every orbit of some action and has been a recent recurring theme in the field of dynamical algebraic combinatorics. For a great overview of this phenomenon, see~\cite{roby:daca,stri:dacp}. The \emph{rowmotion} operator acting on the distributive lattice $\mathcal{J}(P)$ of order ideals of a poset $P$ is one of the operators that has received significant attention in this field. This action is easily described: rowmotion maps an order ideal $I$ to the order ideal generated by $\min(P\setminus I)$.
The name rowmotion we are using here originates from the work of Striker and Williams~\cite{stri:par} and is the most common name for this map now, but this action has been called many different names and studied by numerous authors in a variety of contexts previously, see~\cite{brou:otpo,fond:ooai,pany:oooa,rush:oooo,stan:pae} for a few examples. This simple action often induces rich structure, which has led to discoveries related to order, orbit structure, and other properties on certain families of posets, of which root and minuscule posets (in particular, rectangle posets) are notable examples.

Usual suspects for initially searching for homomesic statistics are those in the span of either the order ideal indicator functions ($\oic{p}$ for $p\in P$) or antichain indicator functions ($\ac{p}$ for $p\in P$). For example, both the order ideal and antichain cardinality statistics, which are simply the sum of all order ideal and antichain indicator functions, respectively, are both known to be homomesic on the $a\times b$ rectangle poset with respective averages $\frac{ab}{2}$ and $\frac{ab}{a+b}$~\cite{prop:hipo}. Let $I_H(P)$ and $A_H(P)$ denote the subspaces of homomesies within $\Span_\R(\{\oic{p}\mid p\in P\})$ and $\Span_\R(\{\ac{p}\mid p\in P\})$, respectively. A natural question is whether one can completely describe $I_H(P)$ and $A_H(P)$. Towards this goal, one can consider two subspaces, $I_T(P)$ and $A_T(P)$, made of those statistics in $I_H(P)$ and $A_H(P)$, respectively, which can be written as linear combinations of so-called toggleability statistics, up to a constant. Namely, if one can write a statistic $f$ as a linear combination of the toggleability statistics and a constant function, then it follows automatically that $f$ is homomesic under rowmotion. Moreover, for certain posets, the derived expression allows one to lift $f$ to the more general piecewise-linear and birational settings and obtain homomesies there~\cite{hopk:mdtc,defa:hvts}.  

This paper focuses on these spaces for fences. Fences have shown up in a variety of contexts due to their ties to cluster algebras, $q$-analogues, and unimodality (see the introduction of~\cite{eliz:rof} for a detailed background). The rowmotion action on fences is an interesting example in dynamical algebraic combinatorics. Its order is not known in general and its birational lift seems to have infinite order, so one expects fewer nice results compared to the minuscule and root posets. However, in their paper, which initiated the study of the dynamical algebraic combinatorics of fences, Elizalde et al.~\cite{eliz:rof} identified several statistics which are homomesic under rowmotion using a correspondence of rowmotion orbits with a certain kind of tilings. Here we show that many of those statistics have the stronger property of being in $I_T(F)$ or $A_T(F)$. 

Our main results give descriptions of the spaces $I_T(F)$ and $A_T(F)$. In Section~\ref{background section}, we give a more formal treatment to the terms loosely defined in the introduction, and we introduce other necessary terminology that will be used later. We exhibit a basis for each of the spaces $I_T(F)$ and $A_T(F)$ separately in Sections~\ref{acts section} and~\ref{oits section}. An interesting corollary of these results is that for a fence $F$ with $n$ elements and $t$ segments, $I_T(F)$ and $A_T(F)$ both have dimension $n-(t-1)$. In~\cite{defa:hvts}, it was conjectured that $\dim(I_T(P)) = \dim(A_T(P))$ for certain root and minuscule posets, but in general, there is no relation between $\dim(I_T(P))$ and $\dim(A_T(P))$. We believe this to be the first infinite family of posets for which the dimensions of the toggleability spaces have been proven.

It should be noted that while the statistics presented in the basis for $A_T(F)$ (Theorem~\ref{ac togg dim}) were shown to be homomesic in~\cite{eliz:rof}, as a corollary of describing $I_T(F)$ we obtain new homomesy results (Theorem~\ref{oic basis equiv const}, part of Corollary~\ref{oi homomesies cor}). 

In Section~\ref{the stat oic section}, we consider the order ideal cardinality statistic $\hat\chi$. While $\hat\chi$ is not in $I_T(F)$, it was conjectured in~\cite{eliz:rof} that $\hat\chi$ is homomesic for certain fences and proven to be homomesic for self-dual fences with three segments.  We connect this conjecture to a conjecture about differences of antichain indicator functions for ``opposite'' elements in the fence. Then we consider those differences and show results about them when the fence is self-dual.

In Section~\ref{sec:pl}, we explain how from the statistics in $I_T(F)$ and $A_T(F)$, we obtain homomesy results for corresponding lifted statistics under piecewise-linear and birational rowmotion. This idea was started in~\cite{hopk:mdtc} and further extended in~\cite{defa:hvts}. However, the main examples were the minuscule and root posets, for which birational rowmotion has finite order. In contrast, piecewise-linear and birational rowmotion on fences seem to have infinite order (except in a special case). However, since the spaces $I_T(F)$ and $A_T(F)$ are relatively big, we get many homomesies even in these settings without explicitly considering the structure of the infinite orbits. 

We end with a discussion in Section~\ref{sec:conclusion section} where we compare the toggleability subspaces with the homomesic subspaces and explain how some of our results imply homomesies under a different action on ideals called promotion.


\section{Background}
\label{background section}

All partially ordered sets (posets) in this paper are assumed to be finite. We call a poset a {\it fence} if it consists of alternating maximal chains called {\it segments} such that
\[
x_1\ltri\cdots\ltri x_{\alpha_1}\rtri\cdots\rtri x_{\alpha_1+\alpha_2}\ltri\cdots,
\]
where $\ltrieq$ is the partial order. We denote the $i$-th segment as $S_i$. We also define $[n]\coloneqq\{1,\dots,n\}.$

An element that covers or is covered by two elements, and hence belongs to two segments, is called \emph{shared}. A shared element is a \emph{peak} if it covers two elements; otherwise, it is a \emph{valley}. Elements that are not shared are \emph{unshared}, and we will denote the set of unshared elements of $S_i$ by $\uns{i}$. The fence in Figure~\ref{fence ex} has three segments $S_1=\{x_1,x_2,x_3\}$, $S_2=\{x_3,x_4,x_5,x_6\}$, and $S_3=\{x_6,x_7\}$. The elements $x_1,x_2,x_4,x_5,$ and $x_7$ are unshared, $x_3$ is a peak, and $x_6$ is a valley.

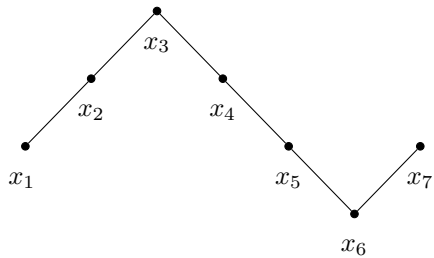
\begin{figure}[!ht]
\centering
\tikzstyle{b} = [circle, draw, fill=black, inner sep=0mm, minimum size=1mm]
\tikzstyle{r} = [circle, draw, color=red, fill=red, inner sep=0mm, minimum size=1mm]
\tikzstyle{line} = [draw, -latex']
\begin{tikzpicture}[xscale=0.875,yscale=0.9,baseline={(0,2.5)}]
\node (1) at (0,0) [b] {};
\node (11) at (-.05,-0.5) {$x_1$};
\node (2) at (1,1) [b] {};
\node (12) at (1,0.5) {$x_2$};
\node (3) at (2,2) [b] {};
\node (13) at (2,1.5) {$x_3$};
\node (4) at (3,1) [b] {};
\node (14) at (3,0.5) {$x_4$};
\node (5) at (4,0) [b] {};
\node (15) at (4,-0.5) {$x_5$};
\node (6) at (5,-1) [b] {};
\node (16) at (5,-1.5) {$x_6$};
\node (7) at (6,0) [b] {};
\node (17) at (6,-0.5) {$x_7$};
\draw (1) -- (2);
\draw (2) -- (3);
\draw (3) -- (4);
\draw (4) -- (5);
\draw (5) -- (6);
\draw (6) -- (7);
\end{tikzpicture}
\caption{An example of a fence, $\breve{F}(3,3,2)$.}
\label{fence ex}
\end{figure}

In general, we refer to a fence through the use of a $t$-tuple of positive integers $\alpha=(\alpha_1,\dots,\alpha_t)$, where $t\geq 2$, $\alpha_1,\alpha_t\geq 2$, and
\[
\alpha_i=\#\uns{i}+1,
\]
for all $i\in [t]$. The fence constructed from this $\alpha$ is denoted $\breve{F}(\alpha)$. It follows from the definition of $\alpha$ that 
\[
\#F=\alpha_1+\alpha_2+\cdots+\alpha_t-1
\]
for any fence $F=\breve{F}(\alpha)$.

An {\it order ideal} of a poset $(P,\ltrieq)$ is a subset $I\subseteq P$ such that if $x\in I$ and $y\ltrieq x$, then $y\in I$ for all $x,y\in P$. From here on, when referring to an order ideal, we will simply say ``ideal". For a subset $S\subseteq P$, we let $\langle S\rangle$ denote the ideal generated by the set $S$. We denote the set of all ideals of a poset by $\mathcal{J}(P)$. The map $\rho:\JJJ(P)\to \JJJ(P)$ called \emph{rowmotion} is defined as
\[
\rho(I)\coloneqq\{x\in P:x\ltrieq y\text{ for some }y\in\min(P\setminus I)\}.
\]
In other words, $\rho$ sends an ideal $I$ to the ideal generated by the minimal elements of the complement of $I$. This is a succinct way to describe rowmotion, but it can also be defined using local maps known as toggles.

For each $p\in P$, define the {\it toggle} $\tau_p:\JJJ(P)\to \JJJ(P)$ as
\[
\tau_p(I)\coloneqq\begin{cases}
I\cup\{p\}&\text{if }p\in\min(P\setminus I),\\
I\setminus\{p\}&\text{if }p\in\max(I),\\
I&\text{otherwise}.
\end{cases}
\]
The toggle operation $\tau_p(I)$ adds $p$ to the ideal $I$ if $p\notin I$ and the result is a valid ideal, removes $p$ if $p\in I$ and removing it results in a valid ideal, and does nothing otherwise. Furthermore, $\tau_p$ and $\tau_{p'}$ commute if $p$ and $p'$ do not share a covering relation.

A \emph{linear extension} of a poset $P$ is a list $p_1,\dots,p_n$ of all elements of $P$ that respects the order of $P$. That is, $p_i\ltrieq p_j$ in $P$ implies $i\leq j$. Cameron and Fon-Der-Flaass showed in~\cite{came:ooar} that rowmotion is equivalent to a composition of toggles, $\rho=\tau_{p_1}\circ\cdots\circ\tau_{p_n}$,
for any linear extension $p_1,\dots,p_n$ of $P$, where the toggles are applied from right to left. See Figure~\ref{rowmotion ex} for an example of a rowmotion orbit.

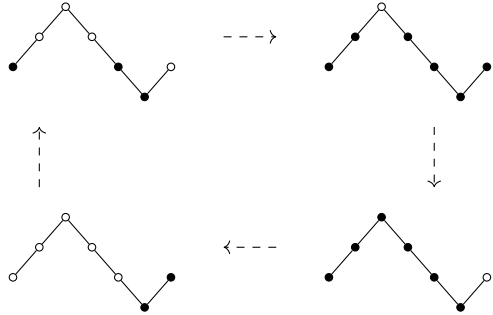
\begin{figure}[!ht]
\centering
\tikzstyle{b} = [circle, draw, fill=black, inner sep=0mm, minimum size=1mm]
\tikzstyle{w} = [circle, draw, color=black, fill=white, inner sep=0mm, minimum size=1mm]
\tikzstyle{line} = [draw, -latex']
\begin{tikzpicture}[xscale=0.35,yscale=0.4,baseline={(0,3)}]
\begin{scope}[shift={(0,4)}]
\node (1) at (0,0) [b] {};
\node (2) at (1,1) [w] {};
\node (3) at (2,2) [w] {};
\node (4) at (3,1) [w] {};
\node (5) at (4,0) [b] {};
\node (6) at (5,-1) [b] {};
\node (7) at (6,0) [w] {};
\draw (1) -- (2);
\draw (2) -- (3);
\draw (3) -- (4);
\draw (4) -- (5);
\draw (5) -- (6);
\draw (6) -- (7);
\draw[->,dashed](8,1) -- (10,1);
\end{scope}

\begin{scope}[shift={(12,4)}]
\node (1) at (0,0) [b] {};
\node (2) at (1,1) [b] {};
\node (3) at (2,2) [w] {};
\node (4) at (3,1) [b] {};
\node (5) at (4,0) [b] {};
\node (6) at (5,-1) [b] {};
\node (7) at (6,0) [b] {};
\draw (1) -- (2);
\draw (2) -- (3);
\draw (3) -- (4);
\draw (4) -- (5);
\draw (5) -- (6);
\draw (6) -- (7);
\draw[->,dashed](4,-2) -- (4,-4);
\end{scope}

\begin{scope}[shift={(12,-3)}]
\node (1) at (0,0) [b] {};
\node (2) at (1,1) [b] {};
\node (3) at (2,2) [b] {};
\node (4) at (3,1) [b] {};
\node (5) at (4,0) [b] {};
\node (6) at (5,-1) [b] {};
\node (7) at (6,0) [w] {};
\draw (1) -- (2);
\draw (2) -- (3);
\draw (3) -- (4);
\draw (4) -- (5);
\draw (5) -- (6);
\draw (6) -- (7);
\draw[->,dashed](-2,1) -- (-4,1);
\end{scope}

\begin{scope}[shift={(0,-3)}]
\node (1) at (0,0) [w] {};
\node (2) at (1,1) [w] {};
\node (3) at (2,2) [w] {};
\node (4) at (3,1) [w] {};
\node (5) at (4,0) [w] {};
\node (6) at (5,-1) [b] {};
\node (7) at (6,0) [b] {};
\draw (1) -- (2);
\draw (2) -- (3);
\draw (3) -- (4);
\draw (4) -- (5);
\draw (5) -- (6);
\draw (6) -- (7);
\draw[->,dashed](1,3) -- (1,5);
\end{scope}

\end{tikzpicture}
\caption[An example of a rowmotion orbit on the fence $\breve{F}(3,3,2)$.]{An example of a rowmotion orbit on the fence $\breve{F}(3,3,2)$, where elements of the ideals are shaded in.}
\label{rowmotion ex}
\end{figure}

Given a poset $P$, an element $p\in P$, and an ideal $I\subseteq P$, we define the {\it order ideal indicator function} $\hat{\chi}_p(I)$ and the {\it antichain indicator function} $\chi_p(I)$ as
\[
\hat{\chi}_p(I)\coloneqq\begin{cases}
1&\text{if }p\in I,\\
0&\text{otherwise}
\end{cases}
\quad\text{and}\quad
\chi_p(I)\coloneqq\begin{cases}
1&\text{if }p\in\max(I),\\
0&\text{otherwise}.
\end{cases}
\]
Related to these are the {\it order ideal cardinality statistic} $\hat{\chi}(I)$ given by
\[
\hat{\chi}(I)\coloneqq\# I=\sum_{p\in P}\hat{\chi}_p(I),
\]
and the {\it antichain cardinality statistic} $\chi(I)$ given by
\[
{\chi}(I)\coloneqq\#\max(I)=\sum_{p\in P}{\chi}_p(I).
\]

We define the {\it toggleability statistic} $T_p:\JJJ(P)\to\R$, for $p\in P$, by
\begin{align*}
T_p(I)&\coloneqq 
\begin{cases}
1&\text{if }p\in\min(P\setminus I),\\
-1&\text{if }p\in\max(I),\\
0&\text{otherwise}.
\end{cases}
\end{align*}
Note that an element $p\in P$ that can be ``toggled in" to the ideal $I$ has $T_p(I)=1$, while an element that can be ``toggled out" has $T_p(I)=-1$. An element that cannot be toggled in or out has $T_p(I)=0$. For ease of notation, given a fence $F$, we also define $T_i\coloneqq T_{x_i}$ for $x_i\in F$.

For example, for the ideal $I=\{x_1,x_2,x_5,x_6\}$ of the fence in Figure~\ref{fence ex}, we have $\max(I)=\{x_2,x_5\}$, $\hat{\chi}(I)=4$, $\chi(I)=2$, $T_2(I)=T_5(I)=-1$, $T_4(I)=T_7(I)=1$, and $T_i(I)=0$ for $i=1,3,6$.

In Section~\ref{sec:pl}, the following alternative way to view the toggleability statistics will be useful: $T_p(I)=T_p^+(I)-T_p^-(I)$, where 
\[
T_p^+(I)\coloneqq\begin{cases}
1&\text{if }p\in\min(P\setminus I),\\
0&\text{otherwise}
\end{cases}
\quad\text{and}\quad
T_p^-(I)\coloneqq\begin{cases}
1&\text{if }p\in\max(I),\\
0&\text{otherwise.}
\end{cases}
\]

A statistic $f:X\to \R$, with $X$ a finite set, is said to be {\it homomesic} if the average of $f$ on every orbit of an invertible operator $T:X\to X$ is equal to some constant $c$, in which case $f$ is said to be {\it $c$-mesic}. The interest in the toggleability statistics, when homomesies are considered, is due to the following fact:

\begin{theorem}[\cite{stri:ttgh}, Lemma 6.2]
For any poset $P$ and any $p\in P$, the toggleability statistic $T_p$ is 0-mesic under rowmotion acting on $\JJJ(P)$.
\end{theorem}
\noindent A simple, intuitive explanation for this theorem is that every element must be toggled in exactly as many times as it is toggled out in a rowmotion orbit. 

Because of the linearity of expectation, any linear combination of 0-mesic statistics is also 0-mesic. So, when trying to prove that a statistic $f$ is homomesic, one can try to take advantage of a technique systematized in~\cite{defa:hvts} by attempting to express $f$ as the sum of a constant $c$ and a linear combination of toggleability statistics. That is, one can try to find suitable constants $c_p$, for $p\in P$, such that
\begin{equation}
\label{togg eq}
f(I)=c+\sum_{p\in P}c_pT_p(I)
\end{equation}
for all $I\in \JJJ(P)$.  If such an identity exists, then the identity implies that $f$ is $c$-mesic. The converse is not true in general. To make the existence of such an identity easier to indicate, we use the notation from~\cite{defa:hvts} and write $f\equiv g$ when there exist constants $c_p\in \R$ such that $f-g=\sum_{p\in P}c_pT_p$. By a slight abuse of notation, for $c\in\R$, we also write $c$ for the function that is identically equal to $c$. Often we will write $f\equiv$ const to denote that $f$ can be written as in~\eqref{togg eq} (and is hence homomesic) without specifying the value of $c$. 

It is straightforward to see that functions which are $\equiv$ const are closed under addition and scalar multiplication and thus form a subspace of the homomesic statistics on $\JJJ(P)$ under rowmotion. We will mainly be concerned with two subspaces, the \emph{antichain toggleability space} $A_T(P)$ and the \emph{order ideal toggleability space} $I_T(P)$, where
\[
A_T(P)\coloneqq \{f:f\in A_H(P) \text{ and } f\equiv\text{const}\}
\]
and
\[
I_T(P)\coloneqq \{f:f\in I_H(P) \text{ and } f\equiv\text{const}\}.
\]
In the following two sections, we will completely describe both of these spaces for a general fence $F$.

\begin{remark}
We summarize here some notation and conventions that will be used throughout the next three sections. We define $\beta_i\coloneqq \#\uns{i}=\alpha_i-1$ for $i\in[t]$. Let $s_{(i,j)}$ be the $j$-th minimal element of $\uns{i}$. When statistics refer to elements of this type, we will often omit the `s', as in $T_{(i,j)}$, $\oic{(i,j)}$, and $\ac{(i,j)}$. Let $s_i$ be the unique shared element in $S_i\cap S_{i+1}$. When statistics refer to elements of this type, we will include the `s', as in $T_{s_i}$, $\oic{s_i}$, and $\ac{s_i}$. We adopt the convention that if an element does not exist in $F$, then the statistics of interest indexed by it (say, $T_{s_0}$, $\oic{s_t}$, $\ac{(i,\alpha_i)}$) are defined to be identically zero on $\JJJ(F)$.
\end{remark}


\section{Antichain Toggleability Space}
\label{acts section}

In this section, we describe a basis for the antichain toggleability space $A_T(F)$ of a general fence $F = \breve{F}(\alpha_1,\dots,\alpha_t)$. Let
\[
\mathcal{B}_A(F) = \bigcup_{i=1}^t\{\alpha_i\ac{x}+\ac{s_{i-1}}+\ac{s_i}\mid x\in\uns{i}\}.
\]
The statistics in $\mathcal{B}_A(F)$ were shown to be homomesic in~\cite{eliz:rof}. Here we first show the stronger result that these statistics are $\equiv$ const, and thus in $A_T(F)$. Since they are linearly independent, this gives a lower bound on $\dim A_T(F)$. Lemma~\ref{glc ac lemma} gives an upper bound on $\dim A_T(F)$. 

\begin{theorem} 
\label{peak and valley anti}
Let $\alpha=(\alpha_1,\dots,\alpha_t)$ with corresponding fence $F=\breve{F}(\alpha)$. Let $i\in[t]$ with $\alpha_i\geq 2$, $x\in \uns{i}$, $v$ be the valley of $S_i$ (if it exists), $p$ be the peak of $S_i$ (if it exists), $y=s_{(i,1)}$, and $z=s_{(i,\beta_{i})}$. Then
\begin{equation}
\alpha_i\chi_x+\chi_v+\chi_p=1-T_v-\sum_{y\ltrieq u\ltrieq x}\#[y,u]T_u+\sum_{x\ltri u\ltrieq z}\#[u,z]T_u.
\label{acts basis}
\end{equation}
\end{theorem}
\begin{proof}
For a given ideal $I$, the left-hand side (LHS) of~\eqref{acts basis} has possible values:
\[
(\alpha_i\chi_x+\chi_v+\chi_p)(I)=\begin{cases}
\alpha_i&\text{if }x\in\max(I),\\
1&\text{if }v\in \max(I),\\
1&\text{if }p\in \max(I),\\
0&\text{if }x,v,p\notin\max(I).
\end{cases}
\]
Next, we compute the value of the right-hand side (RHS) of~\eqref{acts basis} by considering four cases. In evaluating the RHS, it is helpful to note that the relevant toggleability statistics are all associated with elements from the same segment $S_i$. Therefore, for each ideal $I$, at most two of them are nonzero.

Case 1: Let $x\in\max(I)$. Then $T_x=-1$, and there are three subcases to consider based on whether $x$ is covered by an unshared element, $x$ is covered by $p$, or $x$ is not covered by any element.

Case 1a: Assume $x$ is covered by an unshared element $t$. Note that $x\ltri t\ltrieq z$, $[t,z]=(x,z]$, and $T_t=1$, so the RHS simplifies to
\begin{align*}
1-\#[y,x]T_x+\#[t,z]T_t=1+\#[y,x]+\#(x,z]=1+\#[y,z]=1+\beta_i=\alpha_i,
\end{align*}
since $\#[y,z]$ is the number of elements in $\uns{i}$.

Case 1b: Assume $x$ is covered by $p$. In this case, $x=z$ and $T_x=-1$ is the only nonzero toggleability statistic on the RHS, so the RHS reduces to
\[
1-\#[y,x]T_x=\alpha_i.
\]

Case 1c: Assume $x$ is not covered by any element. Note that this case reduces to the previous case since the assumption implies $x=z$, so $T_x=-1$ is still the only nonzero toggleability statistic on the RHS. Thus, the RHS reduces to $\alpha_i$ in this case as well, allowing us to conclude that the RHS always has value $\alpha_i$ if $x\in\max(I)$.

Case 2: Let $v\in\max(I)$, and note that $T_v=-1$. Also note that $y$ covers $v$ and that $T_y=1$ since $y$ can be toggled in. The RHS then simplifies to
\[
1-T_v-\#[y,y]T_y=1.
\]

Case 3: Let $p\in\max(I)$. Note that since $p$ is the only element of $S_i$ that can be toggled, all toggleability statistics on the RHS are zero. Thus, the RHS is equal to 1.

Case 4: Assume $x,v,p\notin\max(I)$. We consider five subcases based on whether the segment $S_i$ has a valley, whether $\max(I)$ contains an element from $S_i$ and, if so, where it is located.

Case 4a: Assume $v$ does not exist, and no elements of $S_i$ are in $I$. In this case, $y$ can be toggled in and $T_y=1$ is the only nonzero toggleability statistic on the RHS, so the RHS is 
\[
1-\#[y,y]T_y=0.
\]

Case 4b: Assume $v$ exists and no elements of $S_i$ are in $I$. In this case, $T_v=1$ since $v$ can be toggled in, and the RHS is $1-T_v=0$.

Case 4c: Assume $v\in I$, but $r\notin\max(I)$ for every $r\in S_i$. In this case, $y$ can be toggled in, so $T_y=1$ and the RHS simplifies to 0. 

Case 4d: Assume $r\in\max(I)$ for some $r\in[y,x)$. In this case, $r$ and the element that covers $r$, say $t$, would result in two nonzero summands in the first sum, with $T_r=-1$ and $T_t=1$. The RHS would then simplify to
\[
1-\#[y,r]T_r-\#[y,t]T_t=0.
\]

Case 4e: Assume $r\in\max(I)$ for some $r\in(x,z]$. Again, we have that $T_r=-1$, and we let $t$ denote the element of $S_i$ that covers $r$. If $t\neq p$, then $T_t$ appears in the second sum and $T_t=1$, so the RHS is
\[
1+\#[r,z]T_r+\#[t,z]T_t=0.
\]
If $t=p$, then $T_t$ does not appear in the second sum and $\#[r,z]=1$, resulting in a RHS value of
\[
1+\#[r,z]T_r=0.
\]
In all five cases, we see the RHS has a value of 0.

In conclusion, the case analysis shows that
\[
\left(1-T_v-\sum_{y\ltrieq u\ltrieq x}\#[y,u]T_u+\sum_{x\ltri u\ltrieq z}\#[u,z]T_u\right)(I)=\begin{cases}
\alpha_i&\text{if }x\in\max(I),\\
1&\text{if }v\in \max(I),\\
1&\text{if }p\in \max(I),\\
0&\text{if }x,v,p\notin\max(I).
\end{cases}
\]
Therefore, the LHS and RHS of~\eqref{acts basis} are equal for every ideal $I\in \JJJ(F)$, giving us the desired result.
\end{proof}

Since statistics which are $\equiv$ const are closed under linear combinations, we have the following corollary of Theorem~\ref{peak and valley anti}, which strengthens some of the homomesy results that were shown in~\cite{eliz:rof}.
\begin{cor}
Let $\alpha=(\alpha_1,\dots,\alpha_t)$ with corresponding fence $F=\breve{F}(\alpha)$.  
\begin{enumerate}
\item If $x,y\in\uns{i}$, then $\ac{x}-\ac{y}\equiv$ 0.
\item If $\alpha_i=2$ for all $i\in[t]$, then $\chi\equiv t/2$.
\end{enumerate}
\label{anti homom cor}
\end{cor}
\begin{proof}
(1) Let $x,y\in\uns{i}$. Then
\[
\ac{x}-\ac{y}=\frac{1}{\alpha_i}(\alpha_i\ac{x}+\ac{s_{i-1}}+\ac{s_i})-\frac{1}{\alpha_i}(\alpha_i\ac{y}+\ac{s_{i-1}}+\ac{s_i}),
\]
and since
\[
\alpha_i\ac{x}+\ac{s_{i-1}}+\ac{s_{i}}\equiv 1\quad\text{and}\quad\alpha_i\ac{y}+\ac{s_{i-1}}+\ac{s_i}\equiv 1,
\]
we have $\ac{x}-\ac{y}\equiv 0$.

(2) Assume $\alpha_i=2$ for all $i\in[t]$. Then
\[
\chi=\sum_{x\in F}\ac{x}=\sum_{i=1}^t\frac{1}{2}(2\ac{(i,1)}+\ac{s_{i-1}}+\ac{s_i}),
\]
and since
\[
2\ac{(i,1)}+\ac{s_{i-1}}+\ac{s_i}\equiv 1
\]
for each $i\in[t]$, we have $\chi\equiv\frac{t}{2}$.
\end{proof}

Theorem~\ref{peak and valley anti} will be used to show that $\dim(A_T(F))\geq\sum_{i=1}^t(\alpha_i-1)$. Lemma~\ref{glc ac lemma} below shows that $\dim(A_T(F))\leq \sum_{i=1}^t(\alpha_i-1)$. Part of the argument in the proof of Lemma~\ref{glc ac lemma} is also used in the proof of Lemma~\ref{glc oic lemma} in Section~\ref{oits section}, so we state it separately as Lemma~\ref{peak 0 coeffs} next. 

For the following results, we will be plugging various ideals into equations of type
\[
c+\sum_{x\in F}c_xT_x=f.
\]
To simplify notation, we will denote by $\eq(S)$ the equation
\[
c+\sum_{x\in F}c_xT_x(\langle S\rangle)=f(\langle S\rangle),
\]
where we will opt to omit the braces around the set $S$ for ease of notation. For example, for the fence $F=\breve{F}(3,3,2)$, which is depicted in Figure~\ref{fence ex}, and the statistic $f=\chi$, we have \[\eq(x_2,x_5):c-c_2+c_4-c_5+c_7=2.\]

For a statement $P$, let 
\[ \1(P) = \begin{cases} 1 & \text{if } P \text { is true,} \\
0 & \text{if } P \text { is false.}
\end{cases}\]

\begin{lemma}
Let $F$ be a fence and let $f \in \Span_\R(\{\ac{x}\mid x\in F\})\cup\Span_\R(\{\oic{x}\mid x\in F\})$. If 
\begin{equation}
f = c+\sum_{x\in F}c_xT_x
\label{f as togg}
\end{equation}
for real constants $c$ and $c_x$, $x \in F$, then $c_p=0$ for every peak $p\in F$. 

If, additionally, $f \in \Span_\R(\{\ac{s_i}\mid i\in[t-1]\})\cup\Span_\R(\{\oic{s_i}\mid i\in[t-1]\})$, then also $c_u=0$ for every unshared $u\in F$.
\label{peak 0 coeffs}
\end{lemma}
\begin{proof}
Let $F=\breve{F}(\alpha_1,\dots,\alpha_t)$, $f \in \Span_\R(\{\ac{s_i}\mid i\in[t-1]\})\cup\Span_\R(\{\oic{s_i}\mid i\in[t-1]\})$, and $p$ be the peak of segments $S_i$ and $S_{i+1}$, for some $i\in[t-1]$. Recall that $\beta_i=\#\uns{i}$ and that $s_{(i,j)}$ is the $j$-th minimal element in $\uns{i}$. Let $c_{(i,j)}$ denote the coefficient corresponding to the element $s_{(i,j)}$ in~\eqref{f as togg}. To find $c_p$, we consider two different cases based on the number of elements in $S_i$.

Case 1: Assume $S_i$ has only two elements. Let $x = \min(S_i)$. Then $x$ is either a valley (if $i\neq 1$) or an unshared element (if $i=1$). In either case,  
\[
\eq(x)-\eq(\emptyset): -2c_x+\1(x\text{ is a valley and $\beta_{i-1}\geq 1$})c_{i-1,1}=C_1,
\]
where $C_1\in \R$ depends on the statistic $f$ and the element $x$. Now, let $y$ denote the element of $S_{i+1}$ covered by $p$. Then 
\[
\eq(x,y)-\eq(y): -2c_x+\1(x\text{ is a valley and $\beta_{i-1}\geq 1$})c_{i-1,1}+c_p=C_1.
\]
Note that the same value $C_1$ appears on the right-hand side since any contribution to the statistic $f$ from the segment $S_{i+1}$ is canceled out when subtracting $\eq(y)$ from $\eq(x,y)$. By subtracting the last two equations, we get $c_p= 0$.

Case 2: Assume $S_i$ has more than two elements. Let $x$ be the element of $S_i$ covered by $p$ and let $z$ be the element covered by $x$. We have that 
\[
\eq(x,z)-\eq(z): -2c_x+c_z+\1(z\text{ is a valley and $\beta_{i-1}\geq 1$})c_{i-1,1}=C_2,
\]
where $C_2 \in \R$ depends on $f, x,$ and $z$. Let $y$ denote the element of $S_{i+1}$ covered by $p$. Then
\[
\eq(x,y,z)-\eq(y,z): -2c_x+c_z+\1(z\text{ is a valley and $\beta_{i-1}\geq 1$})c_{i-1,1}+c_p=C_2,
\]
with the same value $C_2$ on the right-hand side. Subtracting the last two equations gives that $c_p=0$, which shows that $c_p=0$ for every peak $p\in F$.

Now let $f \in \Span_\R(\{\ac{s_i}\mid i\in[t-1]\})\cup\Span_\R(\{\oic{s_i}\mid i\in[t-1]\})$. To find the $c_{(i,j)}$'s, we form a system of equations by subtracting certain equations obtained by closely related ideals. Specifically, fix $i\in[t]$ with $\beta_i\geq 1$. Let $I_0$ be the ideal generated by all the peaks in $F$ in segments $\{1, \dots, i-1\}$ (resp. $\{i+1, \dots, t\}$) if $S_i$ is an up (resp. down) segment.  Let $I_k=\langle I_0\cup s_{(i,k)}\rangle$ for $k\in[\beta_i]$. If $\beta_i > 1$, we use these ideals in~\eqref{f as togg} and subtract consecutive ideals, noting that the right-hand sides will all be zero since the same shared elements appear in each ideal. This process yields the following system of $\beta_i$ equations:
\begin{align}
\eq(I_1)-\eq(I_0)&: -2c_{(i,1)}+c_{(i,2)}=0 \notag\\
\eq(I_2)-\eq(I_1)&: c_{(i,1)}-2c_{(i,2)}+c_{(i,3)}=0 \notag\\
&\vdots \label{system} \\
\eq(I_{\beta_i-1})-\eq(I_{\beta_i-2})&: c_{(i,\beta_i-2)}-2c_{(i,\beta_i-1)}+c_{(i,\beta_i)}=0 \notag\\
\eq(I_{\beta_i})-\eq(I_{\beta_i-1})&: c_{(i,\beta_i-1)}-2c_{(i,\beta_i)}=0.\notag
\end{align}
By using all but the last equation in this system, it is easy to see that we must have
\begin{equation}
c_{(i,k)}=kc_{(i,1)}
\label{mults of min}
\end{equation}
for $k\in[\beta_i]$. Substituting~\eqref{mults of min} into the last equation in~\eqref{system}, we get \[(\beta_i-1)c_{(i,1)}-2\beta_ic_{(i,1)}=0\] and, therefore,
\begin{equation*}
c_{(i,1)}=0.
\end{equation*}
Combining the above equation with~\eqref{mults of min} yields that $c_{(i,k)}=0$ for all $k\in[\beta_i]$.

If $\beta_i=1$, we can similarly obtain
\[ \eq(I_1)-\eq(I_0): -2c_{(i,1)} =0,\]
so in either case,
\begin{equation*}
c_{(i,k)}=0
\end{equation*}
for all $k\in[\beta_i]$. Since $i$ was arbitrary, we conclude that $c_u=0$ for all unshared $u\in F$.
\end{proof}

We now show that $\Span_\R(\{\chi_x\mid x\in F\})$ contains a $(t-1)$-dimensional subspace which has a trivial intersection with $A_T(F)$. Since the total number of elements in a fence is
\[
\#F=\#\{\text{unshared elements of }F\}+\#\{\text{shared elements of } F\}=\left(\sum_{i=1}^t(\alpha_i-1)\right)+(t-1),
\]
this implies that $\dim(A_T(F))\leq \sum_{i=1}^t(\alpha_i-1)$.

\begin{lemma}
Let $\alpha=(\alpha_1,\dots,\alpha_t)$ with corresponding fence $F=\breve{F}(\alpha)$. Then the set $X=\left\{\ac{s_i}\mid 1 \leq i \leq t-1\right\}$ is linearly independent and $\Span_\R(X)\cap A_T(F)=\{0\}.$
\label{glc ac lemma}
\end{lemma}
\begin{proof}
Let $i\in[t-1]$. Note that $\ac{s_i}(\langle s_i\rangle)=1$ and $\ac{s_j}(\langle s_i\rangle)=0$ for $j\neq i$. Therefore, $X$ is linearly independent.

Let $g\in\Span_\R(X)\cap A_T(F)$. Then 
\begin{equation*}
g = b_1\ac{s_1} + \dots + b_{t-1}\ac{s_{t-1}},
\end{equation*}
for some $b_1,\dots, b_{t-1} \in \R$ and there are real numbers $c$ and $c_x$, $x\in F$ such that
\begin{equation}
g=c+\sum_{x\in F}c_xT_x.
\label{ac g as togg}
\end{equation}

We begin by finding $c_x$ for each $x\in F$. By Lemma~$\ref{peak 0 coeffs}$, $c_p=0$ for every peak $p$ and $c_u=0$ for every unshared $u\in F$, so we only need to determine the coefficients corresponding to the valleys. We do this by plugging in different ideals in~\ref{ac g as togg}.

Let $s_i$ be a valley and observe that
\begin{equation*}
\eq(s_i)-\eq(\emptyset):-2c_{s_i}=b_i,
\end{equation*}
so that for a valley $s_i$:
\begin{equation} c_{s_i}=-\frac{b_i}{2}.\label{ac valley coeffs}\end{equation}

We now argue that $b_i=0$ for all $i\in[t-1]$. We consider two cases based on the parity of $i \in[t-1]$. In both cases, we will use the equation
\begin{equation}
\eq(F): c=\sum_{\substack{i=1\\i\text{ odd}}}^{t-1}b_i,
\label{ac eq full}
\end{equation}
which follows from the fact that the set $\max(F)$ consists of the peaks in $F$ and, if $t$ is odd, the maximal element of the last segment.

Case 1: Assume $i$ is odd so that $s_i$ is a peak. The equation $\eq(F\setminus\langle s_i\rangle)$ is
\begin{equation}
c=\sum_{\substack{j=1\\j\text{ odd}\\j\neq i}}^{t-1}b_j.
\label{ac eq full-peak}
\end{equation}
Subtracting~\eqref{ac eq full-peak} from~\eqref{ac eq full} yields $b_i=0$.

Case 2: Let $i\in[t-1]$ with $i$ even. Then $\eq(F\setminus (S_i\cup S_{i+1}))$ is 
\[
c_{s_i}+c=\sum_{\substack{j=1\\j\text{ odd}\\j\neq i-1,i+1}}^{t-1}b_j
\]
and subtracting~\eqref{ac eq full} yields
\begin{equation*}
c_{s_i}=-b_{i-1}-b_{i+1},
\end{equation*}
where $b_i\coloneqq 0$ for $i\notin[t-1]$. 
Since $i$ is even, $i-1$ and $i+1$ are odd and thus zero by either Case 1 or definition if $i+1\notin[t-1]$, so $c_{s_i}=0$. By~\eqref{ac valley coeffs}, this gives that $b_i=0$ for all even $i$, and hence every $i\in[t-1]$. We conclude that $\Span_\R(X)\cap A_T(F)=\{0\}$, as desired.
\end{proof}

\begin{theorem}
Let $\alpha=(\alpha_1,\dots,\alpha_t)$ with corresponding fence $F=\breve{F}(\alpha)$. A basis for $A_T(F)$ is given by
\[
\BBB_A(F)=\bigcup_{i=1}^t\{\alpha_i\ac{x}+\ac{s_{i-1}}+\ac{s_i}\mid x\in\uns{i}\},
\]
and thus 
\[
\dim(A_T(F))=\sum_{i=1}^t(\alpha_i-1).
\]
\label{ac togg dim}
\end{theorem}
\begin{proof}
By Theorem~\ref{peak and valley anti}, $\BBB_A(F) \subset A_T(F)$.
Note that for any $i\in[t]$ and $x \in \uns{i}$, all statistics in  $\mathcal{B}_A(F)$ but $\alpha_i\chi_x+\chi_{s_{i-1}}+\chi_{s_i}$ are zero when evaluated at $I=\langle x \rangle$. Therefore, $\mathcal{B}_A(F)$ is linearly independent and
\[
\dim(A_T(F))\geq\#\BBB_A(F)=\sum_{i=1}^t(\alpha_i-1).
\]
Furthermore, by Lemma~\ref{glc ac lemma},  $\Span_\R(\{\ac{x}\mid x\in F\})$ contains a $(t-1)$-dimensional subspace that trivially intersects $A_T(F)$, which implies
\[
\dim(A_T(F))\leq\#F-(t-1)=\sum_{i=1}^t\alpha_i-1-(t-1)=\sum_{i=1}^t(\alpha_i-1),
\]
which yields the result.
\end{proof}


\section{Order Ideal Toggleability Space}
\label{oits section}

In this section, we use the same approach as in Section~\ref{acts section} to describe $I_T(F)$ of a fence $F=\breve{F}(\alpha_1,\dots,\alpha_t)$ in Theorem~\ref{oic togg dim}.  Recall that $s_{(i,j)}$ is the $j$-th minimal element of $\uns{i}$ and that $\oic{(i,j)}$ and $T_{(i,j)}$ are the order ideal indicator function and toggleability statistic, respectively, corresponding to $s_{(i,j)}$. In Theorem~\ref{oic basis equiv const}, we show that each element of
\[
\BBB_I(F):=\bigcup_{i=1}^t\bigcup_{j=1}^{\beta_i}\{\alpha_i\hat{\chi}_{(i,j)}-j\hat{\chi}_p-(\alpha_i-j)\hat{\chi}_v\mid p\text{ peak of }S_i,~v\text{ valley of }S_i\},
\]
is in $I_T(F)$. The fact that $\BBB_I(F)$ is linearly independent gives a lower bound on $\dim(I_T(F))$. Lemma~\ref{glc oic lemma} gives an upper bound on $\dim(I_T(F))$. 

\begin{theorem}
Let $\alpha=(\alpha_1,\dots,\alpha_t)$ with corresponding fence $F=\breve{F}(\alpha)$. Let $i\in[t]$ with $\alpha_i\geq 2$, $v$ be the valley of $S_i$ (if it exists) and $p$ be the peak of $S_i$ (if it exists). Then, for each $j\in [\beta_i]$,
\begin{multline}
\alpha_i\hat{\chi}_{(i,j)}-j\hat{\chi}_p-(\alpha_i-j)\hat{\chi}_v=\1(S_i\text{ has no valley})(\alpha_i-j)
\\-(\alpha_i-j)\sum_{s_{(i,1)}\ltrieq u\ltrieq s_{(i,j)}}\#\left[s_{(i,1)},u\right]T_u-j\sum_{s_{(i,j)}\ltri u\ltrieq s_{(i,\beta_i)}}\#\left[u,s_{(i,\beta_i)}\right]T_u,
\label{oits basis}
\end{multline}
\label{oic basis equiv const}
\end{theorem}

\begin{proof}
Assume first that $S_i$ has a valley so that $\1(S_i\text{ has no valley})=0$. For a given ideal $I$, the left-hand side (LHS) of~\eqref{oits basis} has possible values:
\begin{equation}
\big(\alpha_i\hat{\chi}_{(i,j)}-j\hat{\chi}_p-(\alpha_i-j)\hat{\chi}_v\big)(I)=\begin{cases}
j&\text{if }v,s_{(i,j)}\in I\text{ and }p\notin I\\
-(\alpha_i-j)&\text{if }v\in I\text{ and }s_{(i,j)},p\notin I,\\
0&\text{if }v,s_{(i,j)},p\notin I\text{ or }v,s_{(i,j)},p\in I.
\end{cases}
\label{oic valley LHS}
\end{equation}

We now compute the value of the right-hand side (RHS) of~\eqref{oits basis} by considering 3 cases.

Case 1: Let $v,s_{(i,j)}\in I$ and $p\notin I$. There are two subcases to consider based on whether or not $s_{(i,j)}\in\max(I)$.

Case 1a: Assume $s_{(i,j)}\in\max(I)$, so $T_{(i,j)}=-1$. If $s_{(i,j)}$ is the maximal element in $\uns{i}$, then $j=\beta_i$ and $T_{(i,j)}$ is the only nonzero toggleability statistic on the RHS, which therefore simplifies to 
\[
-(\alpha_i-\beta_i)\#\left[s_{(i,1)},s_{(i,j)}\right]T_{(i,j)}=j.
\]
If $s_{(i,j)}$ is not the maximal element in $\uns{i}$, then $s_{(i,j)}\ltri s_{(i,j+1)}\ltrieq s_{(i,\beta_i)}$ and $T_{(i,j+1)}=1$, so the RHS reduces to
\begin{equation*}
-(\alpha_i-j)\#\left[s_{(i,1)},s_{(i,j)}\right]T_{(i,j)}-j\#\left[s_{(i,j+1)},s_{(i,\beta_i)}\right]T_{(i,j+1)}=j.
\end{equation*}

Case 1b: Assume $s_{(i,j)}\notin\max(I)$. Then $s_{(i,k)}\in \max(I)$ for some $j<k\leq \beta_i$. Note that $s_{(i,j)}\ltri s_{(i,k)}\ltrieq s_{(i,\beta_i)}$ and $T_{(i,k)}=-1$. If $k=\beta_i$, then $T_{(i,k)}$ is the only nonzero toggleability statistic on the RHS, which simplifies to
\begin{equation*}
-j\#\left[s_{(i,\beta_i)},s_{(i,\beta_i)}\right]T_{(i,j)}=j.
\end{equation*}
If $k\neq\beta_i$, then $s_{(i,k)}\ltri s_{(i,k+1)}\ltrieq s_{(i,\beta_i)}$ and $T_{(i,k+1)}=1$. The RHS reduces to
\begin{equation*}
-j\left(\#\left[s_{(i,k)},s_{(i,\beta_i)}\right]T_{(i,k)}+\#\left[s_{(i,k+1)},s_{(i,\beta_i)}\right]T_{(i,k+1)}\right)=j.
\end{equation*}
In all possible subcases, we see that the RHS has value $j$ if $v,s_{(i,j)}\in I$ and $p\notin I$.

Case 2: Let $v\in I$ and $s_{(i,j)},p\notin I$. If $\uns{i}\cap \max(I)=\emptyset$ or $\uns{i}\cap \max(I)=\{v\}$, then $T_{(i,1)}=1$, and the RHS is thus
\begin{equation*}
-(\alpha_i-j)\#\left[s_{(i,1)},s_{(i,1)}\right]T_{(i,1)}=-(\alpha_i-j).
\end{equation*}
Otherwise, let $s_{(i,k)}\in\uns{i}\cap \max(I)$. Note that $s_{(i,1)}\ltrieq s_{(i,k)}\ltri s_{(i,k+1)}\ltrieq s_{(i,j)}$, $T_{(i,k)}=-1$, and $T_{(i,k+1)}=1$. The RHS is then
\begin{equation*}
-(\alpha_i-j)\left(\#\left[s_{(i,1)},s_{(i,k)}\right]T_{(i,k)}+\#\left[s_{(i,1)},s_{(i,k+1)}\right]T_{(i,k+1)}\right)=-(\alpha_i-j),
\end{equation*}
so the RHS is equal to $-(\alpha_i-j)$ when $v\in I$ and $s_{(i,j)},p\notin I$.

Case 3: Assume $v,s_{(i,j)},p\notin I$ or $v,s_{(i,j)},p\in I$. For both of these assumptions, all of the toggleability statistics on the RHS are zero, so the RHS has value 0.

Summarizing the values of the RHS in the three cases above and comparing them with~\eqref{oic valley LHS} gives the result when $S_i$ has a valley.

For the remainder of the proof, we assume that $S_i$ has no valley. For a given ideal $I$, the LHS has possible values:
\begin{equation}
\big(\alpha_i\hat{\chi}_{(i,j)}-j\hat{\chi}_p\big)(I)=\begin{cases}
\alpha_i-j&\text{if }s_{(i,j)},p\in I\\
\alpha_i&\text{if }s_{(i,j)}\in I\text{ and }p\notin I,\\
0&\text{if }s_{(i,j)},p\notin I.
\end{cases}
\label{oic no valley LHS}
\end{equation} 

If $s_{(i,j)},p\in I$, all toggleability statistics on the RHS are zero, so the RHS reduces to $\alpha_i-j$ since $S_i$ has no valley.

If $s_{(i,j)}\in I$ and $p\notin I$, note that the argument given in Case 1 above gives $j$ as the value of the RHS when ignoring the contribution of $(\alpha_i-j)$ that occurs due to $S_i$ having no valley. The RHS thus has value $j+(\alpha_i-j)=\alpha_i$.

If $s_{(i,j)},p\notin I$, the argument in Case 2 gives $-(\alpha_i-j)$ as the value for the RHS when excluding the contribution of $(\alpha_i-j)$ from $S_i$ not having a valley. The RHS thus has value 0.

Summarizing the above and comparing to~\eqref{oic no valley LHS} yields the result when $S_i$ has no valley, concluding the proof.
\end{proof}

As a consequence of Theorem~\ref{oic basis equiv const}, we obtain many new statistics which are homomesic under rowmotion. In~\cite{eliz:rof}, it was shown that the statistic $k\oic{1,j}-j\oic{1,k}$ is $(k-j)$-mesic. As we show in the following corollary, this homomesy follows from Theorem~\ref{oic basis equiv const} and can further be extended to statistics from the last segment as well. 
\begin{cor}
Let $\alpha=(\alpha_1,\dots,\alpha_t)$ with corresponding fence $F=\breve{F}(\alpha)$. If $S_i$ does not have a valley, then
\[
k\oic{(i,j)}-j\oic{(i,k)}\equiv k-j.
\]
If $S_i$ does not have a peak, then
\[
(\alpha_i-k)\oic{(i,j)}-(\alpha_i-j)\oic{(i,k)}\equiv 0.
\]
\label{oi homomesies cor}
\end{cor}
\begin{proof}
Let $S_i$ be a segment with no valley, and let $p$ be the peak of $S_i$. Observe that
\begin{equation*}
k\oic{(i,j)}-j\oic{(i,k)}=\frac{k}{\alpha_i}(\alpha_i\oic{(i,j)}-j\oic{p})-\frac{j}{\alpha_i}(\alpha_i\oic{(i,k)}-k\oic{p}).
\end{equation*}
By Theorem~\ref{oic basis equiv const}, we have that $\alpha_i\oic{(i,j)}-j\oic{p}\equiv\alpha_i-j$ and $\alpha_i\oic{(i,k)}-k\oic{p}\equiv \alpha_i-k$, so
\[
k\oic{(i,j)}-j\oic{(i,k)}\equiv \frac{k}{\alpha_i}(\alpha_i-j)-\frac{j}{\alpha_i}(\alpha_i-k)
= k-j.
\]
Now, let $S_i$ be a segment with valley $v$ and no peak. Then
\[
(\alpha_i-k)\oic{(i,j)}-(\alpha_i-j)\oic{(i,k)}=\frac{\alpha_i-k}{\alpha_i}(\alpha_i\oic{(i,j)}-(\alpha_i-j)\oic{v})-\frac{\alpha_i-j}{\alpha_i}(\alpha_i\oic{(i,k)}-(\alpha_i-k)\oic{v}).
\]
By Theorem~\ref{oic basis equiv const}, $\alpha_i\oic{(i,j)}-(\alpha_i-j)\oic{v}\equiv 0$ and $\alpha_i\oic{(i,k)}-(\alpha_i-k)\oic{v}\equiv 0$, which gives
\[
(\alpha_i-k)\oic{(i,j)}-(\alpha_i-j)\oic{(i,k)}\equiv 0.
\]
\end{proof}

\begin{lemma}
Let $\alpha=(\alpha_1,\dots,\alpha_t)$ with corresponding fence $F=\breve{F}(\alpha)$. Then the set $Y=\{\oic{s_i}\mid i\in[t-1]\}$
is linearly independent and $\Span_\R(Y)\cap I_T(F)=\{0\}$.
\label{glc oic lemma}
\end{lemma}
\begin{proof}
Let $a_i\in\R$ for all $i\in[t-1]$, such that 
\begin{equation}
a_1\oic{s_1}+\cdots+a_{t-1}\oic{s_{t-1}}=0.
\label{lin comb of f zero oic}
\end{equation}
We proceed by showing $a_i=0$ for all $i\in [t-1]$. 

Observe that $\oic{s_i}(\langle s_i\rangle)\neq 0$ for all $i\in[t-1]$. Let $i\in[t-1]$ with $i$ even so that $s_i$ is a valley. Then $\oic{s_j}(\langle s_i\rangle)=0$ for $j\neq i$, so we must have $a_i=0$. Now assume that $i\in[t-1]$ with $i$ odd so that $s_i$ is a peak. Observe that $\oic{s_j}(\langle s_i\rangle)=0$ for all $j\leq i-2$ and all $j\geq i+2$. Additionally, if $i-1$ and $i+1$ are in $[t-1]$, they are even by the assumption on $i$, and hence $a_{i-1}=a_{i+1}=0$. To satisfy~\eqref{lin comb of f zero oic}, we must have $a_i=0$. This shows $a_i=0$ for all odd $i$, so $a_i=0$ for all $i\in[t-1]$. Therefore, $Y$ is linearly independent.

Let $g\in\Span_\R(Y)\cap I_T(F)$. Then 
\begin{equation*}
g = b_1\oic{s_1} + \dots + b_{t-1}\oic{s_{t-1}},
\end{equation*}
for some $b_1,\dots, b_{t-1} \in \R$ and there are real numbers $c$ and $c_x$, $x\in F$ such that
\begin{equation}
g=c+\sum_{x\in F}c_xT_x.
\label{g as togg}
\end{equation}

We now determine $c_x$ for every $x\in F$. By Lemma~$\ref{peak 0 coeffs}$, $c_p=0$ for every peak $p$ and $c_u=0$ for every unshared $u\in F$, so all that remains is to find the coefficients for the valleys. We do this by plugging in different ideals in~\eqref{g as togg}.

Let $s_i$ be a valley and observe that
\begin{equation*}
\eq(s_i)-\eq(\emptyset):-2c_{s_i}=b_i.
\end{equation*}
so that for any valley $s_i$:
\begin{equation} c_{s_i}=-\frac{b_i}{2}.\label{valley coeffs}\end{equation}

We now argue that $b_i=0$ for all $i\in[t-1]$. We consider two cases based on the parity of $i \in[t-1]$. In both cases, we will use the equation
\begin{equation}
\eq(F): c=\sum_{i=1}^{t-1}b_i.
\label{eq full}
\end{equation}

Case 1: Assume $i$ is odd so that $s_i$ is a peak. The equation $\eq(F\setminus\langle s_i\rangle)$ is
\begin{equation}
c=\sum_{\substack{j=1\\j\neq i}}^{t-1}b_j.
\label{eq full-peak}
\end{equation}
Subtracting~\eqref{eq full-peak} from~\eqref{eq full} yields $b_i=0$.

Case 2: Let $i\in[t-1]$ with $i$ even. Then $\eq((F\setminus (S_i\cup S_{i+1}))\cup\{s_i\})$ is 
\[
-c_{s_i}+c=\sum_{\substack{j=1\\j\neq i-1,i+1}}^{t-1}b_j
\]
and subtracting the above equation from~\eqref{eq full} yields
\begin{equation*}
c_{s_i}=b_{i-1}+b_{i+1},
\end{equation*}
where we again define $b_i\coloneqq 0$ for $i\notin[t-1]$. Since $i$ is even, $i-1$ and $i+1$ are odd and thus zero by either Case 1 or definition if $i+1\notin[t-1]$, so $c_{s_i}=0$. By~\eqref{valley coeffs}, this gives that $b_i=0$ for all even $i$ and hence all $i\in[t-1]$. Thus, $\Span_\R(Y)\cap I_T(F)=\{0\}$, as desired.
\end{proof}

\begin{theorem}
Let $\alpha=(\alpha_1,\dots,\alpha_t)$ with corresponding fence $F=\breve{F}(\alpha)$. A basis for $I_T(F)$ is given by
\[
\BBB_I(F)=\bigcup_{i=1}^t\bigcup_{j=1}^{\beta_i}\{\alpha_i\hat{\chi}_{(i,j)}-j\hat{\chi}_p-(\alpha_i-j)\hat{\chi}_v\mid p\text{ peak of }S_i,v\text{ valley of }S_i\},
\]
and thus 
\[
\dim(I_T(F))=\sum_{i=1}^t(\alpha_i-1).
\]
\label{oic togg dim}
\end{theorem}
\begin{proof}
By Theorem~\ref{oic basis equiv const}, $\BBB_I(F)\subset I_T(F)$. We now show that $\BBB_I(F)$ is linearly independent. 

Let $\xi_{(i,j)}=\alpha_i\hat{\chi}_{(i,j)}-j\oic{p}-(\alpha_i-j)\oic{v}$, where $p$ is the peak of $S_i$ and $v$ is the valley of $S_i$, for $i\in[t]$. Suppose
\begin{equation}
\sum_{i=1}^t\sum_{j=1}^{\beta_i}a_{(i,j)}\xi_{(i,j)}=0,
\label{lc of BI}
\end{equation}
where the $a_{(i,j)}$'s are real constants. Fix $i'\in[t]$ with $\beta_i\geq 1$ and $j'\in[\beta_i]$. Let $I_1$ denote the ideal generated by the element covered by $\unse{i'}{j'}$. If such an element does not exist, let $I_1$ be the empty ideal. Let $I_2=\langle \unse{i'}{j'}\rangle$. Observe that $\xi_{(i,j)}(I_1)=\xi_{(i,j)}(I_2)$ for all $i,j$ such that $(i,j)\neq (i',j')$ and that $\xi_{(i',j')}(I_2)=\xi_{(i',j')}(I_1)+\alpha_i$ since $\unse{i'}{j'}\in I_2$ but $\unse{i'}{j'}\notin I_1$. Evaluating~\eqref{lc of BI} at $I_1$ and then $I_2$ yields two equations:
\begin{equation}
a_{(i',j')}(\xi_{(i',j')})(I_1)+\sum_{i=1}^t\sum_{\substack{j=1\\(i,j)\neq(i',j')}}^{\beta_i}a_{(i,j)}(\xi_{(i,j)})(I_1)=0,
\label{i1 eq}
\end{equation}
and
\[
a_{(i',j')}(\xi_{(i',j')})(I_2)+\sum_{i=1}^t\sum_{\substack{j=1\\(i,j)\neq(i',j')}}^{\beta_i}a_{(i,j)}(\xi_{(i,j)})(I_2)=0.
\]
The second equation is equivalent to
\begin{equation}
a_{(i',j')}((\xi_{(i',j')})(I_1)+\alpha_i)+\sum_{i=1}^t\sum_{\substack{j=1\\(i,j)\neq(i',j')}}^{\beta_i}a_{(i,j)}(\xi_{(i,j)})(I_1)=0.
\label{i2 eq}
\end{equation}
Subtracting~\eqref{i1 eq} from~\eqref{i2 eq} gives 
\[
\alpha_ia_{(i',j')}=0,
\]
so $a_{(i',j')}=0$. Since $(i',j')$ was arbitrary, we must have $a_{(i,j)}=0$ for each $a_{(i,j)}$ in~\eqref{lc of BI}. Therefore $\BBB_I(F)$ is linearly independent and
\[
\dim(I_T(F))\geq \#\BBB_I(F)=\sum_{i=1}^t(\alpha_i-1).
\]
Furthermore, by Lemma~\ref{glc oic lemma}, $\Span(\{\oic{x}\mid x\in F\})$ contains a $(t-1)$-dimensional subspace which intersects trivially with $I_T(F)$. Thus,  
\[
\dim(I_T(F))\leq \#F-(t-1)=\sum_{i=1}^t\alpha_i-1-(t-1)=\sum_{i=1}^t(\alpha_i-1),
\]
which proves equality.
\end{proof}


\section{The Statistic \texorpdfstring{$\hat{\chi}$}{χ}}
\label{the stat oic section}
This section is motivated by a conjecture concerning the order ideal cardinality statistic $\oic{}$ on certain self-dual fences. Let $a^t$ denote the tuple consisting of $t$ elements, all with value $a$.
\begin{conj}[\cite{eliz:rof}]
Let $\alpha=(a^t)$ with $t$ odd and let $F=\breve{F}(\alpha)$. The statistic $\oic{}$ is $n/2$-mesic where $n=\# F$.
\label{oic conj self-dual}
\end{conj}

It is important to note that $\oic{}\not\equiv$ const for this type of fence in general. Thus, the toggleability statistics cannot be used directly to prove this conjectured homomesy. However, we will show how using the toggleability statistics leads to exploring related homomesy statements, which might be helpful in showing Conjecture~\ref{oic conj self-dual}.

The following lemma gives a straightforward way to translate between antichain and order ideal indicator functions for general fences, which we will use for some results in this section and later in the proof of Theorem~\ref{homom equal dim}.

\begin{lemma}
Let $F$ be a fence. Then
\begin{align}
\ac{x}&=\begin{cases}
\oic{x}&\text{if }x\text{ is a peak,}\\
\oic{x}-\oic{y}&\text{if }x \text{ is unshared and $y$ covers $x$,}\\
1-T_x-\oic{x}&\text{if }x \text{ is a valley,}
\end{cases}
\label{ac to oic dict}\\
\intertext{and}
\oic{x}&=\begin{cases}
\ac{x}&\text{if }x\text{ is a peak,}\\
\sum_{y\rtrieq x}\ac{y}&\text{if }x \text{ is unshared,}\\
1-T_x-\ac{x}&\text{if }x \text{ is a valley.}
\end{cases}
\label{oic to ac dict}
\end{align}
\end{lemma}
\begin{proof}
We first show~\eqref{oic to ac dict}. The first two cases of~\eqref{oic to ac dict} follow directly from the fact that $x\in I$ if and only if $y\in\max(I)$ for some $y\rtrieq x$. If $x$ is a valley, we consider three subcases based on whether $x\in I$ and $x\in\max(I)$.

Case 1: Assume $x\notin I$, so that $\oic{x}=\ac{x}=0$. Then $x$ can be toggled in, so $T_x=1$.

Case 2: Assume $x\in\max(I)$, so that $\oic{x}=\ac{x}=1$. Then $x$ can be toggled out, so $T_x=-1$.

Case 3: Assume $x\in I$ but $x\notin\max(I)$. Then $\oic{x}=1$, $\ac{x}=0$, and $T_x=0$.

In each of these cases, one can confirm that when $x$ is a valley, $\oic{x}=1-T_x-\ac{x}$.

Finally,~\eqref{ac to oic dict} follows from~\eqref{oic to ac dict} by M\"obius inversion.
\end{proof}

The above lemma allows us to express $\oic{}$ for $F=\breve{F}(a^t)$, $t$ odd, in an alternative way that involves antichain indicator functions for shared elements.

\begin{theorem}
Let $F=\breve{F}(a^t)$ with $t$ odd. Then
\[
\oic{}\equiv \frac{n}{2}+a\left(\sum_{p:\text{ peak}}\ac{p}-\sum_{v:\text{ valley}}\ac{v}\right).
\]
\label{oic as peaks and valleys}
\end{theorem}
\begin{proof}
We first write the order ideal cardinality as the sum of the individual order ideal indicator functions and then apply~\eqref{oic to ac dict} to express everything in terms of antichain indicator functions and toggleability statistics:
\begin{align*}
\oic{}&=\sum_{p:\text{ peak}}\oic{p}+\sum_{v:\text{ valley}}\oic{v}+\sum_{i=1}^t\sum_{j=1}^{a-1}\oic{(i,j)}\\
&=\sum_{p:\text{ peak}}(2a-1)\ac{p}+\sum_{v:\text{ valley}}(1-T_v-\ac{v})+\sum_{i=1}^t\sum_{j=1}^{a-1}j\ac{(i,j)}.
\end{align*}
By Theorem~\ref{peak and valley anti}, $$\ac{(i,j)}\equiv \frac{1}{a}\left(1-\ac{p_i}-\ac{v_i}\right),$$ where $p_i$ (resp. $v_i$) is the peak (resp. valley) of $S_i$. Substituting this in the last sum yields
\begin{align*}
\oic{}&\equiv\sum_{p:\text{ peak}}(2a-1)\ac{p}+\sum_{v:\text{ valley}}(1-\ac{v})+\frac{1}{a}\sum_{i=1}^t\sum_{j=1}^{a-1}j(1-\ac{p_i}-\ac{v_i})\\
&=\sum_{p:\text{ peak}}(2a-1)\ac{p}+\sum_{v:\text{ valley}}(1-\ac{v})+\frac{1}{a}\sum_{i=1}^t\left(\frac{(a-1)a}{2}\right)(1-\ac{p_i}-\ac{v_i})\\
&=\sum_{p:\text{ peak}}(2a-1)\ac{p}+\sum_{v:\text{ valley}}(1-\ac{v})+\frac{(a-1)}{2}\sum_{i=1}^t(1-\ac{p_i}-\ac{v_i}).
\end{align*}
Since each $p_i$ and $v_i$ appears twice in the last sum,  simplification yields
\[
\oic{}\equiv a\sum_{p:\text{ peak}}\ac{p}+\sum_{v:\text{ valley}}(1-a\ac{v})+\frac{t(a-1)}{2}.
\]
Since the number of valleys is $\frac{t-1}{2}$ and $n=\#F=ta-1$, we obtain
\[
\oic{}\equiv a\sum_{p:\text{ peak}}\ac{p}-a\sum_{v:\text{ valley}}\ac{v}+\frac{t-1}{2}+\frac{t(a-1)}{2}=\frac{n}{2}+a\left(\sum_{p:\text{ peak}}\ac{p}-\sum_{v:\text{ valley}}\ac{v}\right).
\]
\end{proof}

Using Theorem~\ref{oic as peaks and valleys}, one can see that Conjecture~\ref{oic conj self-dual} is equivalent to the following.
\begin{conj}
Let $F=\breve{F}(a^t)$ with $t$ odd. The statistic
\[
\sum_{p:\text{ peak}}\ac{p}-\sum_{v:\text{ valley}}\ac{v}
\]
is 0-mesic.
\label{peaks minus valleys self-dual}
\end{conj}

In fact, we make a stronger conjecture which we have confirmed for $\breve{F}(a^t)$ with $a+t\leq 10$, $\breve{F}(2^9)$, and $\breve{F}(8^3)$:
\begin{conj}
Let $F=\breve{F}(a^t)$ with $t$ odd, and let $i\in[t-1]$. Then
\[
\ac{s_i}-\ac{s_{t-i}}
\]
is 0-mesic.
\label{refined stats self-dual}
\end{conj}

To discuss a potential method for proving the above conjectures, we mention a generalization of toggleability statistics which have been called antichain toggleability statistics~\cite{defa:hvts}. 

Let $P$ be a poset, and let $\mathcal{A}(P)$ be the set of all antichains of $P$. For each $A\in \mathcal{A}(P)$, define the \emph{antichain toggleability statistic} $T_A:\JJJ(P)\to\R$ as
\[
T_A(I)\coloneqq 
\begin{cases}
1&\text{if }A\subseteq\min(P\setminus I),\\
-1&\text{if }A\subseteq\max(I),\\
0&\text{otherwise}.
\end{cases}
\]
Antichain toggleability statistics are far more numerous than the toggleability statistics we have been considering, but an advantage of the former can be seen in the following theorem. 
\begin{theorem}[\cite{defa:hvts}, Theorem 6.2]
For a poset $P$, the space of functions $\JJJ(P)\to\R$ that are 0-mesic under rowmotion is equal to $\Span_\R(\{T_A\mid A\in \mathcal{A}(P)\})$.
\label{antichain togg is 0-mesic}
\end{theorem}

So, if we can prove that a statistic can be written as a constant plus a linear combination of antichain toggleability statistics, then that statistic is homomesic. The difficulty is that the antichain toggleability statistics are not linearly independent in general, so even when one finds such expressions for smaller posets, it's challenging to find one which nicely generalizes to a whole family. In contrast, the toggleability statistics associated with single elements are linearly independent~\cite[Theorem 2.7]{defa:hvts}.

As an example of this, we prove that the statistic $\ac{a}-\ac{2a}$ on the fence $\breve{F}(a,a,a)$, which is a specific instance of the statistics found in Conjectures~\ref{peaks minus valleys self-dual} and~\ref{refined stats self-dual}, is 0-mesic by writing it as a linear combination of antichain toggleability statistics. We note that this homomesy can be derived as a consequence of Theorem~\ref{oic as peaks and valleys} and~\cite[Theorem 4.4]{eliz:rof}, but the proof provided here uses the antichain toggleability statistics directly.

\begin{ex}
Before we state the theorem for general $a$, we provide the example when $a=3$ which should help the reader parse the statement. Consider the fence $F=\breve{F}(3,3,3)$ shown in Figure~\ref{peak minus valley ex}. We have that
\begin{multline*}
\ac{3}-\ac{6}=-T_1-2T_2-3T_3-2T_6-2T_7-T_8+3T_{\{1,6\}}+3T_{\{1,7\}}+3T_{\{2,6\}}+3T_{\{2,7\}}+3T_{\{2,8\}}\\
+3T_{\{3,7\}}+3T_{\{3,8\}}-3T_{\{1,5,7\}}-3T_{\{2,4,8\}}.
\end{multline*}

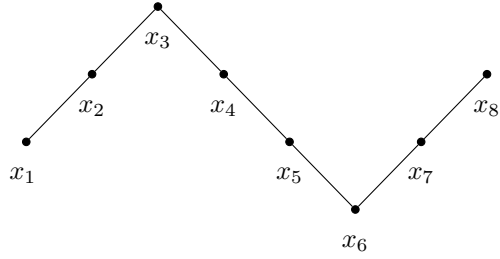
\begin{figure}[!ht]
\centering
\tikzstyle{b} = [circle, draw, fill=black, inner sep=0mm, minimum size=1mm]
\tikzstyle{r} = [circle, draw, color=red, fill=red, inner sep=0mm, minimum size=1mm]
\tikzstyle{line} = [draw, -latex']
\begin{tikzpicture}[xscale=0.875,yscale=0.9,baseline={(0,2.5)}]
\node (1) at (0,0) [b] {};
\node (11) at (-.05,-0.5) {$x_1$};
\node (2) at (1,1) [b] {};
\node (12) at (1,0.5) {$x_2$};
\node (3) at (2,2) [b] {};
\node (13) at (2,1.5) {$x_3$};
\node (4) at (3,1) [b] {};
\node (14) at (3,0.5) {$x_4$};
\node (5) at (4,0) [b] {};
\node (15) at (4,-0.5) {$x_5$};
\node (6) at (5,-1) [b] {};
\node (16) at (5,-1.5) {$x_6$};
\node (7) at (6,0) [b] {};
\node (17) at (6,-0.5) {$x_7$};
\node (8) at (7,1) [b] {};
\node (18) at (7,0.5) {$x_8$};
\draw (1) -- (2);
\draw (2) -- (3);
\draw (3) -- (4);
\draw (4) -- (5);
\draw (5) -- (6);
\draw (6) -- (7);
\draw (7) -- (8);
\end{tikzpicture}
\caption{The fence $\breve{F}(3,3,3)$.}
\label{peak minus valley ex}
\end{figure}
\end{ex}

\begin{theorem}
Let $F=\breve{F}(a,a,a)$ for some $a\geq 2$. Then
\begin{multline}
\ac{a}-\ac{2a}=\sum_{i=1}^{a}(-iT_i)-(a-1)T_{2a}+\sum_{i=1}^{a-1}(-iT_{3a-i})\\+\sum_{i=1}^{a-1}\sum_{j=0}^i(aT_{\{i,2a+j\}})+\sum_{j=1}^{a-1}(aT_{\{a,2a+j\}})+\sum_{i=1}^{a-1}(-aT_{\{i,2a-i,2a+i\}}).
\label{peak minus valley eq for a,a,a}
\end{multline}
\label{peak minus valley for a,a,a}
\end{theorem}

\begin{proof}
Let
\begin{alignat*}{3}
T_{S_1}&=\sum_{i=1}^{a}-iT_i,   &T_{S_3}&=-(a-1)T_{2a}+\sum_{i=1}^{a-1}-iT_{3a-i},\\
T_{\text{pairs}}&=\sum_{i=1}^{a-1}\sum_{j=0}^iaT_{\{i,2a+j\}}+\sum_{j=1}^{a-1}aT_{\{a,2a+j\}},  \qquad &T_{\text{triples}}&=-a\sum_{i=1}^{a-1}T_{\{i,2a-i,2a+i\}}.\\
\end{alignat*}
The claim is that
\[
\ac{a}-\ac{2a}=T_{S_1}+T_{S_3}+T_{\text{pairs}}+T_{\text{triples}}.
\]
Since at most two elements per segment can be toggled for any ideal $I$, most of the (antichain) toggleability statistics are 0 for any given ideal. We show that~\eqref{peak minus valley eq for a,a,a} holds for every ideal by summarizing all possible cases in Table~\ref{all cases}. We use three rectangles as a shorthand to convey the type of the ideal $I$ based on the number of unshared elements in $S_i\cap I$, $1 \leq i \leq 3$. An empty rectangle means there are no unshared elements in $S_i\cap I$, a half-shaded rectangle means that $I$ contains some, but not all, unshared elements from $S_i$, and a shaded rectangle means that $I$ contains all unshared elements from $S_i$. In most cases, the type determines whether the peak and valley are in $I$. When it doesn't, we consider subcases indicated in the disambiguation column. Note that we sometimes intentionally leave cells unsimplified in an effort to give more detail about the contributions of the antichain toggleability statistics. Also note that the final column is always the sum of the third through sixth columns, which proves~\eqref{peak minus valley eq for a,a,a}.
\end{proof}

\begin{table}[!ht]
\centering
\renewcommand{\arraystretch}{1.25}
\begin{tabular}{| c | c | c | c | c | c | c |}
\hline
Unshared elements & Disambiguation & $T_{S_1}(I)$ & $T_{S_3}(I)$ & $T_{\text{pairs}}(I)$ & $T_{\text{triples}}(I)$ & $\ac{a}-\ac{2a}$\\
\hline
\makerecs{er}{er}{er} & $v\notin I$ & $-1$ & $-(a-1)$ & $a$ & $0$ & 0 \\
\hline
\makerecs{er}{er}{er} & $v\in I$ & $-1$ & 0 & $a$ & $-a$ & $-1$ \\
\hline
\makerecs{er}{er}{hr} & & $-1$ & 1 & 0 & 0 & 0 \\
\hline
\makerecs{er}{er}{fr} & & $-1$ & $1$ & $0$ & $0$ & 0 \\
\hline
\makerecs{er}{hr}{er} & & $-1$ & $-(a-1)$ & $a$ & 0 & 0 \\
\hline
\makerecs{er}{hr}{hr} & & $-1$ & $1$ & 0 & 0 & 0 \\
\hline
\makerecs{er}{hr}{fr} & & $-1$ & $1$ & 0 & 0 & 0 \\
\hline
\makerecs{er}{fr}{er} & & $-1$ & $-(a-1)$ & $a$ & $0$ & 0 \\
\hline
\makerecs{er}{fr}{hr} & & $-1$ & $1$ & 0 & 0 & 0 \\
\hline
\makerecs{er}{fr}{fr} & & $-1$ & $1$ & $0$ & $0$ & 0 \\
\hline
\makerecs{hr}{er}{er} & $v\notin I$ & $-1$ & $-(a-1)$ & $a$ & $0$ & 0 \\
\hline
\makerecs{hr}{er}{er} & $v\in I$ & $-1$ & 0 & $-a+a$ & 0 & $-1$ \\
\hline
\makerecs{hr}{er}{hr} & & $-1$ & $1$ & 0 or $-a+a$ & 0 & 0 \\
\hline
\makerecs{hr}{er}{fr} & & $-1$ & $1$ & 0 & 0 & 0 \\
\hline
\makerecs{hr}{hr}{er} & & $-1$ & $-(a-1)$ & $a$ & 0 & 0 \\
\hline
\makerecs{hr}{hr}{hr} & & $-1$ & $1$ & 0 or $-a+a$ & 0 or $a-a$ & 0 \\
\hline
\makerecs{hr}{hr}{fr} & & $-1$ & $1$ & 0 & 0 & 0 \\
\hline
\makerecs{hr}{fr}{er} & & $-1$ & $-(a-1)$ & $a$ & 0 & 0 \\
\hline
\makerecs{hr}{fr}{hr} & & $-1$ & $1$ & 0 or $-a+a$ & 0 & 0 \\
\hline
\makerecs{hr}{fr}{fr} & & $-1$ & $1$ & 0 & 0 & 0 \\
\hline
\makerecs{fr}{er}{er} & $v\notin I$ & $a-1$ & $-(a-1)$ & $0$ & $0$ & 0 \\
\hline
\makerecs{fr}{er}{er} & $v\in I$ & $a-1$ & $0$ & $-a$ & 0 & $-1$ \\
\hline
\makerecs{fr}{er}{hr} & & $a-1$ & $1$ & $-a$ & 0 & 0 \\
\hline
\makerecs{fr}{er}{fr} & & $a-1$ & $1$ & $-a$ & $0$ & 0 \\
\hline
\makerecs{fr}{hr}{er} & & $a-1$ & $-(a-1)$ & 0 & 0 & 0 \\
\hline
\makerecs{fr}{hr}{hr} & & $a-1$ & $1$ & $-a$ & 0 & 0 \\
\hline
\makerecs{fr}{hr}{fr} & & $a-1$ & $1$ & $-a$ & 0 & 0 \\
\hline
\makerecs{fr}{fr}{er} & $p\notin I$ & $-1$ & $-(a-1)$ & $a$ & 0 & 0 \\
\hline
\makerecs{fr}{fr}{er} & $p\in I$ & $a$ & $-(a-1)$ & $0$ & $0$ & 1 \\
\hline
\makerecs{fr}{fr}{hr} & $p\notin I$ & $-1$ & $1$ & $-a+a$ & 0 & 0 \\
\hline
\makerecs{fr}{fr}{hr} & $p\in I$ & $a$ & $1$ & $-a$ & $0$ & 1 \\
\hline
\makerecs{fr}{fr}{fr} & $p\notin I$ & $-1$ & $1$ & $-a$ & $a$ & 0 \\
\hline
\makerecs{fr}{fr}{fr} & $p\in I$ & $a$ & $1$ & $-a$ & $0$ & 1 \\
\hline
\end{tabular}
\caption{Cases for the proof of Theorem~\ref{peak minus valley for a,a,a}.}
\label{all cases}
\end{table}

We note that even though Theorem~\ref{peak minus valley for a,a,a} treats fences with only three segments, the proof involves checking many cases and cannot be generalized as is. However, although we were unsuccessful in finding a general pattern for the coefficients of the antichain toggleability statistics for any of the statistics in Conjectures~\ref{peaks minus valleys self-dual} and~\ref{refined stats self-dual}, the approach may yet prove fruitful with the proper insight.

We have been focusing on differences of antichain indicator functions for ``opposite" peaks and valleys of $\breve{F}(a^t)$, $t$ odd, but we can discuss sums and differences of indicator functions for opposite elements more generally. We will do this next in the expanded context of self-dual fences. We show that, in fact, if all differences of antichain indicator functions for opposite peaks and valleys are 0-mesic for a self-dual fence, then all differences of antichain indicator functions for opposite elements are 0-mesic and all sums of order ideal indicator functions for opposite elements are 1-mesic.

\begin{theorem}
Let $\alpha=(\alpha_1,\dots,\alpha_t)$ be palindromic with $t$ odd, $F=\breve{F}(\alpha)$ be the associated fence, and $n=\# F$. Assume $\ac{s_i}-\ac{s_{t-1}}$ is 0-mesic for all $i\in[t-1]$. Then 
\begin{enumerate}
\item $\ac{k}-\ac{n-k+1}$ is 0-mesic for all $k\in[n]$, and 
\item $\oic{k}+\oic{n-k+1}$ is 1-mesic for all $k\in[n]$.
\end{enumerate}
\label{shared diffs are sufficient}
\end{theorem}
\begin{proof}
(1) If $x_k$ and $x_{n-k+1}$ are shared elements, the result follows by assumption. Assume then that $x_k$ and $x_{n-k+1}$ are unshared elements. Let $x_k\in S_i$, and note that this implies that $x_{n-k+1}\in S_{t-i+1}$. Let $p_j$ (resp. $v_j$) denote the peak (resp. valley) of segment $S_j$, $j \in [t-1]$. By Theorem~\ref{peak and valley anti},
\begin{equation}
\ac{x}\equiv \frac{1}{\alpha_j}(1-\ac{p_j}-\ac{v_j})
\label{ac 3.1 solve}
\end{equation}
for each unshared $x\in S_j$. Combining this with  $\alpha_i=\alpha_{t-i+1}$ (since $\alpha$ is palindromic), we have
\[
\ac{k}-\ac{n-k+1}\equiv \frac{1}{\alpha_i}(1-\ac{p_i}-\ac{v_i})-\frac{1}{\alpha_i}(1-\ac{p_{t-i+1}}-\ac{v_{t-i+1}})=\frac{1}{\alpha_i}(\ac{v_{t-i+1}}-\ac{p_i})+\frac{1}{\alpha_i}(\ac{p_{t-i+1}}-\ac{v_i}).
\]
The combinations of statistics in each set of parentheses are 0-mesic by assumption, so $\ac{k}-\ac{n-k+1}$ is 0-mesic for all $k\in[n]$.

(2) Fix $k\in[n]$. If $x_k$ is a shared element, we can assume without loss of generality that $x_k$ is a peak. Then $x_{n-k+1}$ is a valley, and using~\eqref{oic to ac dict}, we get
\[
\oic{k}+\oic{n-k+1}=1+\ac{k}-\ac{n-k+1}-T_{n-k+1}\equiv 1+\ac{k}-\ac{n-k+1}.
\]
Since $\ac{k}-\ac{n-k+1}$ is 0-mesic, we conclude that $\oic{k}+\oic{n-k+1}$ is 1-mesic.

Suppose now that $x_k$ is an unshared element, say $x_k=\unse{i}{j}$. Then $x_{n-k+1}=\unse{t-i+1}{\alpha_{t-i+1}-j}$. Using~\eqref{oic to ac dict}, ~\eqref{ac 3.1 solve}, and $\alpha_i=\alpha_{t-i+1}$, we have
\begin{align*}
\oic{k}+\oic{n-k+1}&=\ac{p_i}+\sum_{m=j}^{\alpha_i-1}\ac{(i,m)}+\ac{p_{t-i+1}}+\sum_{m=\alpha_i-j}^{\alpha_i-1}\ac{(t-i+1,m)}\\
&\equiv \ac{p_i}+\left(\frac{\alpha_i-j}{\alpha_i}\right)(1-\ac{p_i}-\ac{v_i})+\ac{p_{t-i+1}}+\left(\frac{j}{\alpha_i}\right)(1-\ac{p_{t-i+1}}-\ac{v_{t-i+1}})\\
&=1+\left(\frac{j}{\alpha_i}\right)(\ac{p_i}-\ac{v_{t-i+1}})+\left(\frac{\alpha_i-j}{\alpha_i}\right)(\ac{p_{t-i+1}}-\ac{v_i}).
\end{align*}
Since $\ac{p_i}-\ac{v_{t-i+1}}$ and $\ac{p_{t-i+1}}-\ac{v_i}$ are 0-mesic by assumption, we have that $\oic{k}+\oic{n-k+1}$ is 1-mesic.
\end{proof}

A similar result can be shown in the case where only the sums of order ideal indicator functions for opposite shared elements are assumed to be 1-mesic.
\begin{theorem}
Let $\alpha=(\alpha_1,\dots,\alpha_t)$ be palindromic with $t$ odd, $F=\breve{F}(\alpha)$ be the associated fence, and $n=\# F$. Assume $\oic{s_i}+\oic{s_{t-i}}$ is 1-mesic for all $i\in[t-1]$. Then we have that 
\begin{enumerate}
\item $\oic{k}+\oic{n-k+1}$ is 1-mesic for all $k\in[n]$, and 
\item $\ac{k}-\ac{n-k+1}$ is 0-mesic for all $k\in[n]$.
\end{enumerate}
\label{shared sums are sufficient}
\end{theorem}
\begin{proof}
(1) If $x_k$ and $x_{n-k+1}$ are shared elements, the result follows by assumption. Assume then that $x_k$ and $x_{n-k+1}$ are unshared elements, say $x_k=s_{(i,j)}$ and hence $x_{n-k+1}=s_{(t-i+1,\alpha_{t-i+1}-j)}$. We consider cases based on whether one of $x_k$, $x_{n-k+1}$ is in $S_1$. 

Case 1: Assume one of $x_k$ or $x_{n-k+1}$ is in $S_1$. Without loss of generality, assume $x_k\in S_1$. Note that $S_1$ has no valley and $S_t$ has no peak. By Theorem~\ref{oic basis equiv const} and the palindromicity of $\alpha$, we have
\[
\oic{(i,j)}\equiv \frac{1}{\alpha_1}(\alpha_1-j+j\oic{p_1})\quad \text{and}\quad \oic{(t-i+1,\alpha_1-j)}\equiv\frac{1}{\alpha_1}(j\oic{v_{t-i+1}}),
\]
so that 
\[
\oic{k}+\oic{n-k+1}\equiv \frac{1}{\alpha_i}(\alpha_i-j+j\oic{p_i})+\frac{1}{\alpha_i}(j\oic{v_{t-i+1}})=\frac{\alpha_i-j}{\alpha_i}+\frac{j}{\alpha_i}(\oic{p_i}+\oic{v_{t-i+1}}).
\]
We have that $\oic{p_i}+\oic{v_{t-i+1}}$ is 1-mesic by assumption, so $\oic{k}+\oic{n-k+1}$ is 1-mesic in this case.

Case 2: Assume $x_k$ and $x_{n-k+1}$ are not in $S_1$. Then, by Theorem~\ref{oic basis equiv const},
\begin{align*}
\oic{k}+\oic{n-k+1}&\equiv \frac{1}{\alpha_i}(j\oic{p_i}+(\alpha_i-j)\oic{v_i})+\frac{1}{\alpha_i}((\alpha_i-j)\oic{p_{t-i+1}}+j\oic{v_{t-i+1}})\\
&=\frac{j}{\alpha_i}(\oic{p_i}+\oic{v_{t-i+1}})+\frac{\alpha_i-j}{\alpha_i}(\oic{v_i}+\oic{p_{t-i+1}}).
\end{align*}
The combinations of statistics in each set of parentheses are 1-mesic by assumption, so $\oic{k}+\oic{n-k+1}$ is 1-mesic for all $k\in[n]$.

(2) Fix $k\in[n]$. If $x_k$ is a shared element, we can assume without loss of generality that $x_k$ is a peak. Then $x_{n-k+1}$ is a valley, and using~\ref{ac to oic dict} yields
\[
\ac{k}-\ac{n-k+1}=\oic{k}+\oic{n-k+1}-1+T_{n-k+1}\equiv\oic{k}+\oic{n-k+1}-1.
\]
Since $\oic{k}+\oic{n-k+1}$ is 1-mesic, we conclude that $\ac{k}-\ac{n-k+1}$ is 0-mesic.

If $x_k$ is unshared, say $x_k=\unse{i}{j}$, then $x_{n-k+1}=\unse{t-i+1}{\alpha_{t-i+1}-j}$. Suppose that $x_k$, and hence $x_{n-k+1}$, is on an up-segment (the other case is similar). We consider two cases based on whether or not $x_{k}$ or $x_{n-k+1}$ is in $S_1$. In both cases, we show the proof assuming that the covering elements $x_{k+1}$ and $x_{n-k+2}$ are unshared elements (if they exist). Considering the cases when one (or both) of the covering elements is a peak is very similar, so we omit the details.

Case 1: Assume $x_{k}$ or $x_{n-k+1}$ is in $S_1$. Without loss of generality, we can assume $x_{k}\in S_1$. Using the palindromicity of $\alpha$,~\ref{ac to oic dict}, and that
\[
\oic{(i,j)}\equiv \frac{1}{\alpha_i}(\1(S_i \text{ has no valley})(\alpha_i-j)+j\oic{p_i}+(\alpha_i-j)\oic{v_i})
\]
by Theorem~\ref{oic basis equiv const}, we have
\begin{align*}
\ac{k}-\ac{n-k+1}&=\oic{k}-\oic{k+1}-\oic{n-k+1}+\oic{n-k+2}\\
&\equiv \frac{1}{\alpha_i}-\frac{1}{\alpha_i}\oic{p_i}-\frac{1}{\alpha_i}\oic{v_{t-i+1}}\\
&=\frac{1}{\alpha_i}-\frac{1}{\alpha_i}(\oic{p_i}+\oic{v_{t-i+1}}).\\
\end{align*}
Since $\oic{p_i}+\oic{v_{t-i+1}}$ is 1-mesic by assumption, we have that $\ac{k}-\ac{n-k+1}$ is 0-mesic.

Case 2: Assume $x_{k}$ and $x_{n-k+1}$ are not in $S_1$. Then
\begin{align*}
\ac{k}-\ac{n-k+1}&=\oic{k}-\oic{k+1}-\oic{n-k+1}+\oic{n-k+2}\\
&\equiv \frac{1}{\alpha_i}\ac{v_i}-\frac{1}{\alpha_i}\ac{p_i}-\frac{1}{\alpha_i}\ac{v_{t-i+1}}+\frac{1}{\alpha_{i}}\ac{p_{t-i+1}}\\
&=\frac{1}{\alpha_i}(\oic{v_i}+\oic{p_{t-i+1}})-\frac{1}{\alpha_i}(\oic{p_i}+\oic{v_{t-i+1}}),\\
\end{align*}
which implies $\ac{k}-\ac{n-k+1}$ is 0-mesic.
\end{proof}

The statistics in Theorems~\ref{shared diffs are sufficient} and~\ref{shared sums are sufficient} were considered in~\cite{eliz:rof} and were shown to be homomesic under additional palindromic constraints related to the realization of rowmotion orbits as tilings. By combining the two previous theorems, we get as a corollary that these sets of homomesies are equivalent for self-dual fences.

\begin{cor}
Let $\alpha=(\alpha_1,\dots,\alpha_t)$ be palindromic with $t$ odd, $F=\breve{F}(\alpha)$ be the associated fence, and $n=\# F$. The statistics $\ac{k}-\ac{n-k+1}$ are 0-mesic for all $k\in[n]$ if and only if the statistics $\oic{k}+\oic{n-k+1}$ are 1-mesic for all $k\in[n]$.
\label{rof iff}
\end{cor}

While we haven't found a counterexample or been able to prove the homomesies in Conjectures~\ref{peaks minus valleys self-dual} and~\ref{refined stats self-dual}, we can prove a weaker result which is similar in flavor to a result in~\cite{eliz:rof}. An \emph{order filter} of a poset $P$ is a set $U$ where $x\in U$ and $y\geq x$ implies $y\in U$. Let $\UUU(P)$ denote the set of order filters of $P$. Let $\nabla:\UUU(P)\to\AAA(P)$ be the map defined as
\[
\nabla(U)\coloneqq\{x\in U\mid x\in\min(U)\},
\]
and let $\Delta:\JJJ(P)\to\AAA(P)$ be the map
\[
\Delta(I)\coloneqq\{x\in I\mid x\in\max(I)\}.
\]
Additionally, let $c:P\to P$ be the complement map given by $c(S)=P\setminus S$ for any subset $S\subseteq P$. Note that rowmotion is a composition of the previous maps (and their inverses):
\[
\rho=\Delta^{-1}\circ\nabla\circ c.
\]
In fact, different versions of rowmotion on antichains and order filters have been studied which are compositions of the maps defined above and their inverses, see for example~\cite{jose:atar}.

Let $P$ be a self-dual poset and $\kappa: P\to P$ an order-reversing bijection. We define the \emph{ideal complement} (with respect to $\kappa$) of $I$:
\[
\overline{I}\coloneqq\kappa\circ c(I).
\]

The following lemma connects $I$ and $\overline{I}$ via rowmotion.
\begin{lemma}[\cite{eliz:rof}]
Let $P$ be self-dual and fix an order-reversing bijection $\kappa:P\to P$. Then for all $I\in\JJJ(P)$,
\[
\rho^{-1}(\overline{I})=\overline{\rho(I)},
\]
where the ideal complements are with respect to $\kappa$.
\end{lemma}
To prove this lemma, it was shown that the following diagram commutes:

\begin{equation}
\begin{tikzcd}
I\ar{r}{c} & U\ar{r}{\nabla}\ar{d}{\kappa} & A\ar{r}{\Delta^{-1}}\ar{d}{\kappa} & J\ar{d}{\kappa} & \\
& \overline{I}\ar{r}{\Delta} & B\ar{r}{\nabla^{-1}} & V\ar{r}{c} & K
\end{tikzcd}
\label{comm diagram}
\end{equation}
We will use the above diagram to state a similar result. Let $\kappa$ be an order-reversing bijection on a self-dual poset $P$. Given an ideal $I\in\JJJ(P)$, define 
\[
I'\coloneqq \Delta^{-1}\circ\kappa\circ\Delta(I).
\]

\begin{lemma}
Let $P$ be self-dual and fix an order-reversing bijection $\kappa:P\to P$. If $\kappa$ is an involution, then for all $I\in\JJJ(P)$,
\[
\rho^{-1}(I')=(\rho(I))'.
\]
\label{prime rho rel}
\end{lemma}
\begin{proof}
Note that
\[
\rho^{-1}(I')= c\circ\nabla^{-1}\circ\kappa\circ\Delta(I),
\]
and 
\[
(\rho(I))'=\Delta^{-1}\circ\kappa\circ\nabla\circ c(I).
\]
Using~\eqref{comm diagram} and the fact $\overline{\overline{I}}=I$  since $\kappa$ and $c$ are commuting involutions, we see that
\[
\rho^{-1}(I')=\overline{I}=(\rho(I))',
\]
which completes the proof.
\end{proof}
For the statement of the next result, it will be helpful to define $f(\OOO)\coloneqq\sum_{I\in\OOO}f(I)$, where $\OOO$ is an orbit of rowmotion and $f$ is a statistic on ideals.
\begin{cor}
Let $P$ be a self-dual poset with $n=\# P$ and suppose there is an order-reversing involution $\kappa:P\to P$. Let $I\in\JJJ(P)$ and fix $\ell\in[n]$.
\begin{enumerate}
\item If $I,I'\in\OOO$ for some orbit $\OOO$, then
\[
(\ac{\ell}-\ac{\kappa({\ell})})(\OOO)=0.
\]
\item If $I\in\OOO$ and $I'\in\OOO'$ for some orbits $\OOO$ and $\OOO'$ with $\OOO\neq\OOO'$, then $\#\OOO=\#\OOO'$ and
\[
(\ac{\ell}-\ac{\kappa(\ell)})(\OOO)+(\ac{\ell}-\ac{\kappa(\ell)})(\OOO')=0.
\]
\end{enumerate}
\label{ac - kappa ac}
\end{cor}
\begin{proof}
(1) Let $I,I'\in \OOO$, then observe that for some $J=\rho^j(I)\in\OOO$, we have by Lemma~\ref{prime rho rel}
\[
J'=(\rho^j(I))'=\rho^{-j}(I')\in\OOO.
\]
That is, $\OOO$ can be partitioned into pairs of ideals $\{I,I'\}$ with $I\neq I'$ and singletons $\{I\}$ with $I=I'$. Furthermore, since $\ell\in \max(I)$ if and only if $\kappa(\ell)\in \max(I')$, the value of $\ac{\ell}-\ac{\kappa(\ell)}$ is 0 on each of the pairs and singletons and hence 0 on the entire orbit.

(2) Let $\OOO=\{I_1,\dots,I_m\}$. Similar reasoning as in the proof of (1) shows that $\OOO'=\{I_1',\dots,I_m'\}$. The value of $\ac{\ell}-\ac{\kappa(\ell)}$ is 0 on each of the pairs $\{I_i,I'_i\}$, which implies the result.
\end{proof}

Let $P$ be self-dual with an order-reversing involution $\kappa:P\to P$. Note that self-dual fences have such a map but there exist self-dual posets for which none of the order-reversing bijections are involutions~\cite[Exercise 8, Chapter 3]{stan:ecv1}. Consider the group generated by the action of $\rho$ and the map $I\mapsto I'$, and call the orbits of this action \emph{dihedral group orbits}. By the proof of Corollary~\ref{ac - kappa ac}, all orbits will be either a rowmotion orbit or a union of two rowmotion orbits. We have the following homomesy result as a consequence.
\begin{theorem}
Let $P$ be self-dual with $n=\# F$ and an order-reversing involution $\kappa: P\to P$, and let $\ell\in[n]$. Then $\ac{\ell}-\ac{\kappa(\ell)}$ is 0-mesic on dihedral group orbits.
\end{theorem}


\section{Homomesies for Piecewise-linear and Birational Rowmotion} \label{sec:pl}

Earlier in the paper, we alluded to the fact that proving a statistic is in the toggleability space for (combinatorial) rowmotion, which we consider in this paper, yields results for both piecewise-linear and birational rowmotion. We briefly describe these maps and the process for lifting statistics to these realms here.

In~\cite{eins:cpla}, Einstein and Propp introduced piecewise-linear rowmotion through a composition of toggles. Let $\R^P$ denote the set of functions $\pi:P\to\R$. Let $\widehat{P}$ denote the poset obtained from $P$ by adding a minimal element $\widehat{0}$ and maximal element $\widehat{1}$, and fix parameters $\alpha^\pl$, $\omega^\pl\in\R$. We also view any $\pi\in\R^P$ as a function $\pi:\widehat{P}\to\R$ via the conventions that $\pi(\wh{0})=\alpha^\pl$ and $\pi(\wh{1})=\omega^\pl$. To denote that an element $x\in P$ is covered by $y\in P$, we write $x\lessdot y$. For $\pi\in\R^P$ and $p\in P$, define the \emph{piecewise-linear toggle} $\tau_p^{\text{PL}}$ to be
\[
\tau_p^{\pl}(\pi)(p')\coloneqq\begin{cases}
\pi(p')&\text{if }p\neq p',\\
\min\{\pi(r):p\lessdot r\in \widehat{P}\}+\max\{\pi(r):p\gtrdot r\in \widehat{P}\}-\pi(p)&\text{if }p=p'.
\end{cases}
\]
Note that the toggles $\tau_p^\pl$ and $\tau_{p'}^\pl$ commute if $p$ and $p'$ do not share a covering relation. As with combinatorial rowmotion, we define \emph{piecewise-linear rowmotion} $\rho^{\text{PL}}:\R^P\to\R^P$ as
\[
\rho^{\text{PL}}\coloneqq\tau_{p_1}^{\text{PL}}\circ\cdots\circ\tau_{p_n}^{\text{PL}},
\]
where $p_1,\dots,p_n$ is any linear extension of $P$. 

It should be noted that when given $I\in \JJJ(P)$, if we take $\pi_I\in\R^P$, where $\pi_I$ is the indicator function of the complement of $I$, and set $\alpha^\pl=0$ and $\omega^\pl=1$, piecewise-linear toggling of $\pi_I$ corresponds to combinatorial toggling of $I$. Therefore, any results that hold for piecewise-linear rowmotion also hold for combinatorial rowmotion, and the process of going from piecewise-linear to combinatorial rowmotion is called \emph{specialization}.

Einstein and Propp~\cite{eins:cpla} also introduced a further generalization of rowmotion called birational rowmotion via ``detropicalizing" the piecewise-linear expressions by replacing $+$ with $\times$ and $\max$ with $+$. Let $\R^P_{>0}$ denote the set of functions $\pi:P\to\R_{>0}$. View any $\pi\in\R^P_{>0}$ as a function $\pi:\widehat{P}\to\R_{>0}$ via the convention that $\pi(\widehat{0})=\alpha^\bi$ and $\pi(\widehat{1})=\omega^\bi$. For $p\in P$, define the \emph{birational toggle} $\tau_p^{\text{B}}:\R^P_{>0}\to\R^P_{>0}$ by 
\[
\tau_p^{\text{B}}(\pi)(p')\coloneqq\begin{cases}
\pi(p')&\text{if }p\neq p',\\
\displaystyle{\frac{\displaystyle\sum_{p\gtrdot r\in\widehat{P}}\pi(r)}{\pi(p)\cdot\displaystyle\sum_{p\lessdot r\in\widehat{P}}\pi(r)^{-1}}}&\text{if }p=p',
\end{cases}
\]
and, as before, $\tau_p^{\text{B}}$ and $\tau_{p'}^{\text{B}}$ commute if $p$ and $p'$ do not share a covering relation.
We define \emph{birational rowmotion} $\rho^{\text{B}}:\R^P_{>0}\to\R^P_{>0}$ by
\[
\rho^{\text{B}}\coloneqq \tau_{p_1}^{\text{B}}\circ\cdots\circ\tau_{p_n}^{\text{B}},
\]
where $p_1,\dots,p_n$ is any linear extension of $P$. Any results found in the birational realm give results in the piecewise-linear realm by replacing $+$ with $\max$ and $\times$ with $+$ everywhere in expressions in a process known as \emph{tropicalization}.

The surprising fact is that though both piecewise-linear and birational rowmotion can induce infinite orbits even on finite posets, there are cases when the order is finite (see, for example,~\cite{grin:ipob2} and~\cite{grin:ipob1}).

Fence posets do not seem to be one of the families which enjoy finite order under these generalizations of rowmotion in general. In fact, the only fence we were able to find with finite birational order through computer experimentation was the fence $F=\breve{F}(a,a)$ for $a\geq 2.$ This case can be explained through a couple of results in the literature. Call a poset $P$ \emph{graded} if all chains from a minimal element of $P$ to a maximal element of $P$ have the same length. In~\cite{grin:ipob2}, Grinberg and Roby consider \emph{skeletal} posets, which are inductively defined and include graded rooted forests (of which $\breve{F}(a,a)$ is an example). For skeletal posets, they prove that the order of birational rowmotion coincides with the order of combinatorial rowmotion. The order of combinatorial rowmotion on the fence $F=\breve{F}(a,a)$ was proven in~\cite{eliz:rof} to be $a(a+1)$, and so we have the following result.
\begin{theorem}
Let $\alpha=(a,a)$, $a\geq 2$, and let $F=\breve{F}(\alpha)$ be the corresponding fence. The order of $\rho^{B}$ on $F$ is $a(a+1)$.
\end{theorem}

In addition to the order, we can say a bit more about homomesies in the piecewise-linear and birational realms as well for the statistics in $A_T(F)$ and $I_T(F)$. We note that our original definition of homomesy needs to be altered since $\R^P$ and $\R_{>0}^P$ are not finite sets. We thus consider the limits of the time averages (arithmetic and geometric means). More explicitly, a statistic $f:\R^P\to\R$ is said to be homomesic under piecewise-linear rowmotion if, for every $\pi\in\R^P$, 
\[
\lim_{n\to\infty}\frac{1}{n}\sum_{i=0}^{n-1}f((\rho^{\text{PL}})^i(\pi))=c,
\]
for some $c\in R$. In this case, $f$ is said to be $c$-mesic. Analogously, a statistic $f:\R_{>0}^P\to\R$ is said to be multiplicatively homomesic under birational rowmotion if, for every $\pi\in\R_{>0}^P$,
\[
\lim_{n\to\infty}\left(\prod_{i=0}^{n-1} f((\rho^{\text{B}})^i(\pi))\right)^{\frac{1}{n}}=c,
\]
for some constant $c$. The statistic $f$ is said to be multiplicatively $c$-mesic in this case.

We now describe how to lift toggleability statistics to these realms and show that they are indeed still homomesic. Much of what follows was shown in~\cite{hopk:mdtc} and~\cite{defa:hvts}, so we omit some details. We begin with the piecewise-linear toggleability statistics. Let $P$ be a finite poset. For $p\in P$, define $T_p^{+,\pl},T_p^{-,\pl},T_p^{\pl}:\R^P\to\R$ by
\begin{align*}
T_p^{+,\pl}(\pi)&\coloneqq\pi(p)-\max\{\pi(r):p\gtrdot r\in\widehat{P}\},\\
T_p^{-,\pl}(\pi)&\coloneqq\min(\pi(r):p\lessdot r\in\widehat{P}\}-\pi(p), \\
T_p^{\pl}(\pi)&\coloneqq T_p^{+,\pl}(\pi)-T_p^{-,\pl}(\pi).
\end{align*}
Note that $T_p^{\pl}$ specializes to $T_p$. 
\begin{lemma}[\cite{hopk:mdtc}, Lemma 4.31 and Remark 4.33]
For any $\pi\in\R^P$ and $p\in P$, $T_p^{\textup{PL}}$ is 0-mesic, i.e.,
\[
\lim_{n\to\infty}\frac{1}{n}\sum_{i=0}^{n-1}T_p^{\textup{PL}}((\rho^{\textup{PL}})^i(\pi))=0.
\]
\label{pl togg 0mesic}
\end{lemma}

For the birational case, we define $T_p^{+,\bi},T_p^{-,\bi},T_p^{\bi}:\R^P_{>0}\to\R_{>0}$ by
\begin{align*}
T_p^{+,\bi}(\pi)&\coloneqq\frac{\pi(p)}{\sum_{p\gtrdot r\in\widehat{P}}\pi(r)},\\
T_p^{-,\bi}(\pi)&\coloneqq\frac{1}{\pi(p)\cdot\sum_{p\lessdot r\in\widehat{P}}\pi(r)^{-1}},\\
T_p^{\bi}(\pi)&\coloneqq T_p^{+,\bi}(\pi)/T_p^{-,\bi}(\pi).
\end{align*}

In~\cite{hopk:mdtc}, Hopkins shows that $T_p^\bi$ is 1-mesic for posets that have only finite orbits under birational rowmotion. In~\cite[Remark 4.4]{defa:hvts}, it is stated that one can extend this result to infinite orbits, but a proof is not written anywhere, so we give one here. For this, we need to establish bounds for $T_p^{+,\bi}$. The first observation is trivial, but useful: $\rho^\bi(\pi)(\wh{0})=\pi(\wh{0})$ and $\rho^\bi(\pi)(\wh{1})=\pi(\wh{1})$ for any $\pi\in \R_{>0}^P$. This is easily seen since no toggle changes the labels at $\wh{0}$ or $\wh{1}$~\cite{grin:ipob1}. An additional observation is that the sum
\[
\gamma(\pi)\coloneqq\sum_{p,r\in\wh{P},~p\lessdot r}\frac{\pi(p)}{\pi(r)}
\]
is preserved under birational rowmotion~\cite{eins:cpla,grin:broa}.

The proof of the following lemma, which uses the above fact, is due to Darij Grinberg.
\begin{lemma}
Let $\pi\in\R_{>0}^P$ and $p\in P$. Then there exist $C_1, C_2\in\R_{>0}$ such that for $n\geq 0$,
\[
C_1\leq(\rho^\textup{\bi})^n(\pi)(p)\leq C_2.
\]
\label{bounds on labels}
\end{lemma}
\begin{proof}
Note that each $\frac{\pi(p)}{\pi(r)}$, where $p,r\in \wh{P}$ with $p\lessdot r$, is bounded from above by $\gamma(\pi)$. Since $\gamma(\pi)$ is preserved under birational rowmotion, we have that $\gamma((\rho^\bi)^n(\pi)) = \gamma(\pi)$. Thus, 
\begin{equation}
\frac{(\rho^\bi)^n(\pi)(p)}{(\rho^\bi)^n(\pi)(r)} \leq \gamma(\pi)
\label{upper bound}
\end{equation}
for all $p,r\in\wh{P}$ with $p\lessdot r$ and any $n\in\N$. Then for $p \in P$, applying~\eqref{upper bound} along a saturated chain $p \lessdot r \lessdot \dots \lessdot \wh 1$ of length $k$, we get \[(\rho^\bi)^n(\pi)(p)\leq (\rho^\bi)^n(\pi)(r)\gamma(\pi)\leq\cdots\leq \omega^\bi\gamma(\pi)^{k} \leq \max\{ \omega^\bi, \omega^\bi\gamma(\pi)^{M}\},\] where $M$ is the length of the longest chain in $P$. Similarly, applying~\eqref{upper bound} on a saturated chain $p \gtrdot r \gtrdot \dots \gtrdot \wh 0$, one obtains a lower bound on $(\rho^\bi)^n(\pi)(p)$ independent of $n$.
\end{proof}

It is not hard to show that Lemma~\ref{bounds on labels} yields nonzero upper and lower bounds on $T_p^{+,\bi}((\rho^\bi)^n(\pi))$ which are independent of $n$. Additionally, we have that for any $\pi\in\R_{>0}^P$ and $p\in P$, $T_p^{+,\bi}(\pi)=T_p^{-,\bi}(\rho^\bi(\pi))$~\cite[Proof of Lemma 4.42]{hopk:mdtc}. These observations allow us to prove the multiplicative homomesy for the birational toggleability statistics.

\begin{lemma}
For any $\pi\in\R_{>0}^P$ and $p\in P$, $T_p^\textup{\bi}$ is multiplicatively 1-mesic.
\label{bi togg 1mesic}
\end{lemma}
\begin{proof}
Let $\pi\in\R_{>0}^P$ and $p\in P$. Since $T_p^{+,\bi}(\pi)=T_p^{-,\bi}(\rho^\bi)(\pi)$, we have that, for any $n\in\N$,
\[
\prod_{i=0}^{n-1}T_p^\bi(\pi)=\frac{T_p^{+,\bi}(\rho^\bi)^{n-1}(\pi)}{T_p^{-,\bi}(\pi)}
\]
due to cancellations in the numerator and denominator. Since $T_p^{+,\bi}(\rho^\bi)^{n-1}(\pi)$ is bounded by nonzero constants, we have 
\[
\lim_{n\to\infty}\left(\prod_{i=0}^{n-1}T_p^\bi(\pi)\right)^{\frac{1}{n}}=\lim_{n\to\infty}\left(\frac{T_p^{+,\bi}(\rho^\bi)^{n-1}(\pi)}{T_p^{-,\bi}(\pi)}\right)^{\frac{1}{n}}=1,
\]
as desired.
\end{proof}

Since every element in a fence covers and is covered by at most two elements, every statistic in either the order ideal or antichain toggleability space can be lifted  to statistics in the piecewise-linear and birational realms via corresponding linear combinations of lifted toggleability statistics and a constant. Moreover, these lifted statistics still enjoy homomesy in their respective realms. We give a couple of examples of this below, but first we need some definitions. Note that $\ac{p}=T_p^-$.

Let $p\in P$ and define
\begin{alignat*}{3}
&\oic{p}^{\text{PL}}(\pi)\coloneqq \omega^\pl-\pi(p), \qquad&\ac{p}^{\text{PL}}(\pi)&\coloneqq\min\{\pi(r):p\lessdot r\in\wh{P}\}-\pi(p),\\
&\oic{p}^{\text{B}}(\pi)\coloneqq\frac{\omega^\bi}{\pi(p)}, &\ac{p}^{\text{B}}(\pi)&\coloneqq\frac{1}{\pi(p)\cdot\sum_{p\lessdot r\in\wh{P}}\pi(r)^{-1}}.
\end{alignat*}
If a statistic $f:\JJJ(P)\to\R$ is of the form
\begin{equation}
f = \sum_{p\in P}(a_p\oic{p}+b_p\ac{p}),
\label{comb stat}\notag
\end{equation}
set 
\begin{equation}
f^\pl\coloneqq\sum_{p\in P}a_p\oic{p}^\pl+b_p\ac{p}^\pl 
\label{pl stat}\notag
\end{equation}
and
\begin{equation}
f^\bi\coloneqq\prod_{p\in P}(\oic{p}^\bi)^{a_p}\cdot(\ac{p}^\bi)^{b_p}.
\label{bi stat}\notag
\end{equation}

If $P$ is such that every element covers or is covered by at most two elements, by~\cite[Lemma 4.9]{defa:hvts}, $f^\pl$ and $f^\bi$ are well-defined.  Moreover, by~\cite[Theorem 4.11]{defa:hvts}, if $f\equiv c$, then $f^\pl\equiv^\pl c(\omega^\pl-\alpha^\pl)$ and $f^\bi\equiv^\bi\left(\frac{\omega^{\bi}}{\alpha^{\bi}}\right)^c$, where $\equiv^\pl$ and $\equiv^\bi$ are defined analogously to $\equiv$ but with lifted toggleability statistics (and division as opposed to subtraction in the birational case). In particular, by Lemmas~\ref{pl togg 0mesic} and~\ref{bi togg 1mesic}, this means that statistics in the spaces $I_T(F)$ and $A_T(F)$ of a fence $F$ automatically lift to statistics in the piecewise-linear and birational realms \emph{which are still homomesic.}

We end this section with a couple of examples of lifted statistics.

\begin{ex} 
Recall from Theorem~\ref{peak and valley anti} that for any $x\in\breve{S}_i$, if $p$ is the peak of $S_i$ and $v$ is the valley of $S_i$, we have
\[
\alpha_i\ac{x}+\ac{v}+\ac{p}\equiv 1.
\]
Then, for any $\pi\in\R^P$,
\[
\lim_{n\to\infty}\frac{1}{n}\sum_{i=0}^{n-1}\left(\alpha_i\ac{x}^{\pl}(\rho^\pl)^i(\pi)+\ac{v}^{\pl}(\rho^\pl)^i(\pi)+\ac{p}^{\pl}(\rho^\pl)^i(\pi)\right)=\omega^\pl-\alpha^\pl.
\]
In the birational case, for any $\pi\in\R_{>0}^P$, we have
\[
\lim_{n\to\infty}\left(\prod_{i=0}^{n-1}\ac{x}^{\bi}(\rho^\bi)^i(\pi)^{\alpha_i}\cdot\ac{v}^{\bi}(\rho^\bi)^i(\pi)\cdot\ac{p}^{\bi}(\rho^\bi)^i(\pi) \right)^{\frac{1}{n}}=\left(\frac{\omega^\bi}{\alpha^\bi}\right).
\]
\end{ex}

\begin{ex}
From Theorem~\ref{oic basis equiv const}, in the case where $S_i$ has no valley, we have that
\[
\alpha_i\oic{(i,j)}-j\oic{p}\equiv \alpha_i-j,
\]
where $p$ is the peak of $S_i$. Then, for any $\pi\in\R_{>0}^P$,
\[
\lim_{n\to\infty}\frac{1}{n}\sum_{i=0}^{n-1}\left(\alpha_i\oic{(i,j)}^{\pl}(\rho^\pl)^i(\pi)-j\oic{p}^{\pl}(\rho^\pl)^i(\pi)\right)=(\alpha_i-j)(\omega^\pl-\alpha^\pl).
\]
In the birational case, for any $\pi\in\R_{>0}^P$, we have
\[
\lim_{n\to\infty}\left(\prod_{i=0}^{n-1}\oic{(i,j)}^{\bi}(\rho^\bi)^i(\pi)^{\alpha_i}\cdot\oic{p}^{\bi}(\rho^\bi)^i(\pi)\right)^{\frac{1}{n}}=\left(\frac{\omega^\bi}{\alpha^\bi}\right)^{\alpha_i-j}.
\]
\end{ex}


\section{Discussion}
\label{sec:conclusion section}

\subsection{Homomesic subspaces} Recall that $I_H(F)$ and $A_H(F)$ are the subspaces of $\langle\oic{i}\rangle\coloneqq\Span_\R(\{\oic{i}\mid i\in[n]\})$ and $\langle\ac{i}\rangle\coloneqq\Span_\R(\{\ac{i}\mid i\in[n]\})$, where $n$ is the size of the fence $F$, which consist of homomesic statistics and contain the subspaces $I_T(F)$ and $A_T(F)$, respectively. Since our results show that $\dim(I_T(F))=\dim(A_T(F))$, it is a natural question to ask how these subspaces are different and whether we can describe the dimensions of $I_H(F)$ and $A_H(F)$. While we cannot give the complete answers to these questions for a general fence $F$, one can easily show that the two spaces have the same dimension.
\begin{theorem}
For any fence $F$, we have
\[
\dim(I_H(F))=\dim(A_H(F)).
\]
\label{homom equal dim}
\end{theorem}
\begin{proof}
Let $H$ be the space of all statistics on $F$ which are homomesic under rowmotion. By Theorem~\ref{antichain togg is 0-mesic}, $H$ is the span (over $\R$) of all antichain toggleability statistics and the constant functions. Then
\[
\dim(I_H(F))=\dim(H\cap\langle\oic{i}\rangle)=\dim(H)+\dim(\langle\oic{i}\rangle)-\dim(H+\langle\oic{i}\rangle)
\]
and
\[
\dim(A_H(F))=\dim(H\cap\langle\ac{i}\rangle)=\dim(H)+\dim(\langle\ac{i}\rangle)-\dim(H+\langle\ac{i}\rangle).
\]
Note that $\dim(\langle\oic{i}\rangle)=\dim(\langle\ac{i}\rangle) = n$, so it is sufficient to prove that $\dim(H+\langle\oic{i}\rangle)=\dim(H+\langle\ac{i}\rangle)$. Using~\eqref{ac to oic dict} and~\eqref{oic to ac dict}, and the fact that the toggleability statistics are homomesic, one can see that $H+\langle\oic{i}\rangle = H+\langle\ac{i}\rangle$, so the result follows.
\end{proof}

We note that a similar proof to the one above can be used to show $\dim(I_T(F))=\dim(A_T(F))$ since~\eqref{ac to oic dict} and~\eqref{oic to ac dict} only involve toggleability statistics but does not determine the dimension of the spaces, nor their bases.

For a fence $F$ with $t$ segments, we proved that $\dim(I_T(F))=\dim(A_T(F))=n-(t-1)$, which gives an upper bound on the codimensions of $I_H(F)$ and $A_H(F)$. Based on our computer experiments for fences with up to 7 segments, we conjecture that the codimension of the homomesic space can take a value within the full range $[0,t-1]$. Specifically,

\begin{conj}
Given $t\geq 2$ and $0\leq i\leq t-1$, there exists a fence $F$ with $t$ segments such that
\[
\dim(I_H(F))-\dim(I_T(F))=i.
\]
\label{homom all dims conj}
\end{conj}

\begin{ques}
For which fences do $\langle\oic{i}\rangle$ and $\langle\ac{i}\rangle$ not contain additional homomesic statistics, i.e., $I_H(F)=I_T(F)$ and $A_H(F)=A_T(F)$?
\label{homom equals togg ques}
\end{ques}

\noindent Some data relevant to this question is summarized in Table~\ref{data table}. Since the homomesy phenomenon becomes interesting only if there is more than one orbit under the considered action, we also include information about how many small fences have only one orbit under rowmotion.

\begin{table}[!ht]
\renewcommand{\arraystretch}{1.5}
\begin{tabular}{| c | c | c | c | c |}
\hline
$t$ & $n$ & $\#$ fences & $\#$ fences with single orbit & $\#\{F\mid\dim(I_H(F))=\dim(I_T(F))\}$\\
\hline
3 & $\leq 20$ & 969 & 234 $(\approx 24.15\%)$ & 40 $(\approx 4.13\%)$ \\
\hline
4 & $\leq 20$ & 3876 & 346 $(\approx 8.93\%)$ & 188 $(\approx 4.85\%)$ \\
\hline
5 & $\leq 15$ & 2002 & 54 $(\approx 2.70\%)$ & 74 $(\approx 3.70\%)$ \\
\hline
6 & $\leq 15$ & 3003 & 16 $(\approx 0.53\%)$ & 218 $(\approx 7.26\%)$ \\
\hline
7 & $\leq 15$ & 3432 & 2 $(\approx 0.06\%)$ & 472 $(\approx 13.75\%)$ \\
\hline
\end{tabular}
\caption{Data for Question~\ref{homom equals togg ques} for small fences.}
\label{data table}
\end{table}

\subsection{Promotion} While we have focused on the action of rowmotion on order ideals of a poset in this paper, which involves toggling in the reverse order of a linear extension, other actions that are compositions of toggles have also been considered, see for example~\cite{came:ooar,stri:par,stri:ttgh}. One action, in particular, is promotion, which involves toggling ``columns" from left to right. With the way we have been drawing the fences, each element is in its own ``column", so performing promotion amounts to toggling the elements from left to right. That is, if $F$ is a fence with $n=\# F$, then the promotion map $\pro:\JJJ(F) \rightarrow \JJJ(F)$ is defined as 
\[
\pro\coloneqq\tau_n\circ\cdots\circ\tau_1.
\]

Striker and Williams~\cite{stri:par} showed that, for any rc-poset, rowmotion and promotion are conjugate elements in the toggle group: the group generated by all of the order ideal toggles. This result was later generalized to posets $P$ with an $n$-dimensional lattice projection~\cite{dilk:rioo} (an order- and rank-preserving map  $\pi: P \rightarrow \mathbb{Z}^n$). In particular, this implies that promotion and rowmotion have the same orbit structure for such posets.

Enjoying the same orbit structure does not necessarily translate into having the same homomesies, though. For example, the toggleability statistics which have been heavily utilized to show homomesy results in this paper are not generally 0-mesic under promotion. However, Einstein and Propp~\cite{eins:cpla} were the first to use a process called \emph{recombination} which relates the promotion and rowmotion orbits to show that statistics which are linear combinations of order ideal indicator functions are homomesic under rowmotion on the rectangle poset if and only if they are homomesic under promotion. This idea was later generalized to posets with an $n$-dimensional lattice projection by Vorland~\cite{vorl:hipo}. In particular, Theorem 56 from~\cite{vorl:hipo} applies to fences and we illustrate it with an example.

\begin{ex} For this example, we draw the fence $F$ rotated clockwise by $\pi/4$, so that the element $1$ is in the top left corner (Figure~\ref{fig:recomb}) and the notion of columns is clear. Suppose $F$ has $k$ columns. The recombination, $\RR(I)$, of an ideal $I$ of $F$ is the union of the $j$-th columns of $\rho^{j-1}(I)$, $1\leq j \leq k$. Then $\pro(\RR(I)) = \RR(\rho(I))$. 
\end{ex}

\begin{figure}[h]
\centering
\tikzstyle{b} = [circle, draw, fill=black, inner sep=0mm, minimum size=1mm]
\tikzstyle{w} = [circle, draw, color=black, fill=white, inner sep=0mm, minimum size=1mm]
\tikzstyle{r} = [circle, draw, color=red, inner sep=0mm, minimum size=3mm]
\tikzstyle{u} = [circle, draw, color=blue, inner sep=0mm, minimum size=3mm]
\begin{tikzpicture}[xscale=0.5,yscale=0.5,baseline={(0,3)}]
\begin{scope}[shift={(0,4)}]
\node (1) at (0,2) [b] {};
\node (7) at (0,2) [r] {};
\node (2) at (1,2) [b] {};
\node (3) at (2,2) [w] {};
\node (4) at (2,1) [b] {};
\node (5) at (2,0) [b] {};
\node (6) at (3,0) [b] {};
\draw (1) -- (2);
\draw (2) -- (3);
\draw (3) -- (4);
\draw (4) -- (5);
\draw (5) -- (6);
\draw[->,dashed](4,1) -- (5,1);
\node (10) at (4.4,1.6) {$\rho$};
\end{scope}

\begin{scope}[shift={(6,4)}]
\node (1) at (0,2) [b] {};
\node (8) at (0,2) [u] {};
\node (2) at (1,2) [b] {};
\node (7) at (1,2) [r] {};
\node (3) at (2,2) [b] {};
\node (4) at (2,1) [b] {};
\node (5) at (2,0) [b] {};
\node (6) at (3,0) [w] {};
\draw (1) -- (2);
\draw (2) -- (3);
\draw (3) -- (4);
\draw (4) -- (5);
\draw (5) -- (6);
\draw[->,dashed](4,1) -- (5,1);
\node (10) at (4.4,1.6) {$\rho$};
\end{scope}

\begin{scope}[shift={(12,4)}]
\node (1) at (0,2) [w] {};
\node (2) at (1,2) [w] {};
\node (7) at (1,2) [u] {};
\node (3) at (2,2) [w] {};
\node (4) at (2,1) [w] {};
\node (5) at (2,0) [b] {};
\node (6) at (3,0) [b] {};
\draw (1) -- (2);
\draw (2) -- (3);
\draw (3) -- (4);
\draw (4) -- (5);
\draw (5) -- (6);
\draw (2,1) ellipse (0.5cm and 1.6cm) [red];
\draw[->,dashed](4,1) -- (5,1);
\node (10) at (4.4,1.6) {$\rho$};
\end{scope}

\begin{scope}[shift={(18,4)}]
\node (1) at (0,2) [b] {};
\node (2) at (1,2) [w] {};
\node (3) at (2,2) [w] {};
\node (4) at (2,1) [b] {};
\node (5) at (2,0) [b] {};
\node (6) at (3,0) [w] {};
\node (7) at (3,0) [r] {};
\draw (1) -- (2);
\draw (2) -- (3);
\draw (3) -- (4);
\draw (4) -- (5);
\draw (5) -- (6);
\draw (2,1) ellipse (0.5cm and 1.6cm) [blue];
\draw[->,dashed](4,1) -- (5,1);
\node (10) at (4.4,1.6) {$\rho$};
\end{scope}

\begin{scope}[shift={(24,4)}]
\node (1) at (0,2) [b] {};
\node (2) at (1,2) [b] {};
\node (3) at (2,2) [w] {};
\node (4) at (2,1) [w] {};
\node (5) at (2,0) [b] {};
\node (6) at (3,0) [b] {};
\node (7) at (3,0) [u] {};
\draw (1) -- (2);
\draw (2) -- (3);
\draw (3) -- (4);
\draw (4) -- (5);
\draw (5) -- (6);
\end{scope}

\begin{scope}[shift={(9.5,0)}]
\node (1) at (0,2) [b] {};
\node (7) at (0,2) [r] {};
\node (2) at (1,2) [b] {};
\node (8) at (1,2) [r] {};
\node (3) at (2,2) [w] {};
\node (4) at (2,1) [w] {};
\node (5) at (2,0) [b] {};
\node (6) at (3,0) [w] {};
\node (9) at (3,0) [r] {};
\draw (1) -- (2);
\draw (2) -- (3);
\draw (3) -- (4);
\draw (4) -- (5);
\draw (5) -- (6);
\draw (2,1) ellipse (0.5cm and 1.6cm) [red];
\draw[->,dashed](4,1) -- (5,1);
\node (10) at (4.4,1.6) {$\pro$};
\end{scope}

\begin{scope}[shift={(15.5,0)}]
\node (1) at (0,2) [b] {};
\node (7) at (0,2) [u] {};
\node (2) at (1,2) [w] {};
\node (8) at (1,2) [u] {};
\node (3) at (2,2) [w] {};
\node (4) at (2,1) [b] {};
\node (5) at (2,0) [b] {};
\node (6) at (3,0) [b] {};
\node (9) at (3,0) [u] {};
\draw (1) -- (2);
\draw (2) -- (3);
\draw (3) -- (4);
\draw (4) -- (5);
\draw (5) -- (6);
\draw (2,1) ellipse (0.5cm and 1.6cm) [blue];
\end{scope}
\end{tikzpicture}
\caption{Illustration of recombination:  $\pro(\RR(I)) = \RR(\rho(I))$.}
\label{fig:recomb}
\end{figure}
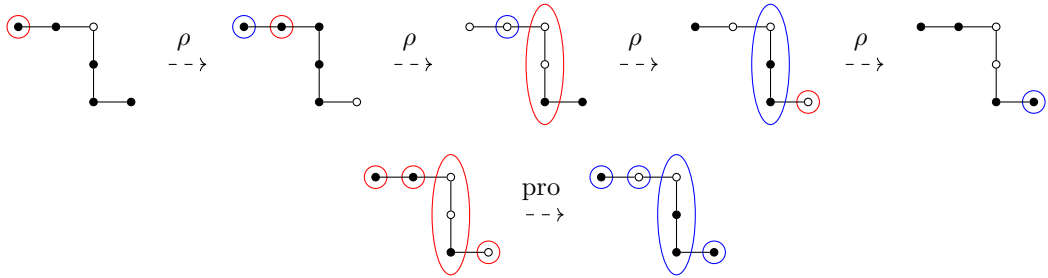

So, as a corollary of Theorem~\ref{oic basis equiv const}, we have the following result.
    
\begin{theorem} Let $\alpha=(\alpha_1,\dots,\alpha_t)$ with corresponding fence $F=\breve{F}(\alpha)$. The statistics in
\[
\Span_\R\left(\bigcup_{i=1}^t\bigcup_{j=1}^{\beta_i}\{\alpha_i\hat{\chi}_{(i,j)}-j\hat{\chi}_p-(\alpha_i-j)\hat{\chi}_v\mid p\text{ peak of }S_i,v\text{ valley of }S_i\}\right)
\]
are homomesic under promotion.
\end{theorem}

\end{document}